\documentclass[12pt,epsfig,amsfonts]{amsart} 
\usepackage{amsmath,amsthm,amssymb,amscd,epsfig,color}
\usepackage{graphicx}
\usepackage{mathrsfs}

\newtheorem*{maintheorem}{Main Theorem}

\newtheorem{prop}{Proposition}[section]
\newtheorem{lemma}[prop]{Lemma}

\newtheorem{theorem}[prop]{Theorem}

\theoremstyle{remark}
\newtheorem{remark}[prop]{Remark}

\numberwithin{equation}{section}

\begin{document}

\author{Johannes Jaerisch and Hiroki Takahasi}

\address{Graduate School of Mathematics, Nagoya University,
Furocho, Chikusaku, Nagoya, 464-8602, JAPAN} 
\email{jaerisch@math.nagoya-u.ac.jp}
\urladdr{http://www.math.nagoya-u.ac.jp/~jaerisch/}

\address{Keio Institute of Pure and Applied Sciences (KiPAS), Department of Mathematics,
Keio University, Yokohama,
223-8522, JAPAN} 
\email{hiroki@math.keio.ac.jp}
\urladdr{\texttt{http://www.math.keio.ac.jp/~hiroki/}}

\subjclass[2010]{11K55, 37D25, 37D35, 37D40}
\thanks{{\it Keywords}: dimension theory; mixed multifractal spectra; neutral periodic points; non-uniformly expanding Markov maps, backward continued fraction expansion, Bowen-Series maps}

\title[Mixed multifractal spectra of Birkhoff averages]{
Mixed multifractal spectra of Birkhoff averages\\ for non-uniformly expanding one-dimensional Markov maps
with countably many branches
}
 \date{\today}
 
 \maketitle
   \begin{abstract} 
  For a Markov map of an interval or the circle
   with countably many branches and 
   finitely many neutral periodic points,
  we establish conditional variational formulas for the mixed multifractal spectra of Birkhoff averages of countably many
  observables,
  in terms of the Hausdorff dimension of invariant probability measures. 
Using our results, 
we are able to exhibit new fractal-geometric results for  backward continued fraction expansions of real numbers, answering in particular a question of Pollicott. Moreover, we establish formulas for  multi-cusp winding spectra for the Bowen-Series maps associated with finitely generated free Fuchsian groups with parabolic elements.
   \end{abstract}

 
  \section{Introduction}

       Borel's normal number theorem states that Lebesgue almost every real number has the property
       that the limiting frequency of each digit in the decimal expansion is $1/10$. 
       It is therefore natural to ask for the Hausdorff dimension of the set of real numbers
       with a different limiting behavior. This problem was investigated by
       Besicovitch \cite{Bes34} and Eggleston \cite{Egg49} in the 30s and 40s, who computed the Hausdorff dimension of the set of
       real numbers whose digits have limiting frequency given 
       by any prescribed probability vector.

 This number-theoretic problem can be generalized as follows.
             Given a dynamical system on a metric space $X$, 
      a sequence $A_1,\ldots, A_k$ of pairwise disjoint subsets of $X$, and
       a probability vector $\boldsymbol\alpha=(\alpha_1,\ldots,\alpha_k)$,
          is it possible to describe the Hausdorff dimension of the set of initial points whose orbits visit each $A_j$ with
       frequency $\alpha_j$? One can also consider a frequency of visits to a countably infinite number of 
       subsets, and even more generally, Birkhoff averages of 
       a countably infinite number of observables.
       The problem is to describe
       {\it mixed multifractal spectra
        of Birkhoff averages}, or simply {\it mixed Birkhoff spectra} of countably many observables.

Mixed Birkhoff spectra combine different local characteristics which depend simultaneously on 
families of observables.
     For certain maps having a Markov structure, several results are known which describe mixed 
Birkhoff spectra:
             Fan and Feng \cite{FanFen00}, and Olsen \cite{Ols03,Ols04}
            on finite shift spaces;
                Fan et al. \cite{FLMW10} on the Hausdorff dimension of Besicovitch-Eggleston sets on infinite shift spaces;
               Fan et al. \cite{FanJorLiaRam16,FLM10} on uniformly expanding Markov interval maps with infinitely many branches.

  In this paper, we 
  describe mixed Birkhoff spectra for
   non-uniformly expanding one-dimensional Markov maps with a countable (finite or infinite) number of branches. 
    The lack of uniform expansion is due to the existence of neutral periodic points.
  Various types of results in different settings on  dimension theory for such
       {\it parabolic systems} 
        (i.e., maps with neutral periodic points) are fairly in abundance. We refer to        \cite{GelRam09,Iom10,JorRam11,KesStr04,Nak00,PolWei99,Yur02} for results on Lyapunov spectra, 
        \cite{MauUrb00,MauUrb03,Urb96} on the Hausdorff dimension of limit sets,  
        \cite{Cli13,Hof10,JJOP10} on Birkhoff spectra of observables not necessarily related to Lyapunov exponents, and  \cite{TakVer03} on entropy spectra.
          However, to our knowledge,
        there is no known result on the mixed Birkhoff spectra of parabolic systems,
        in spite of its potential of broad applications.


\subsection{Statement of the main result}
Let $M=[0,1]$ or $M=\mathbb{S}^1$, and let $S$ be a 
   countable set.
   {\it A Markov map} is a map $f\colon\varDelta\to M$ such that the following holds:

 \begin{itemize}
 
\item[(M0)]
   There exists a family $\{\varDelta_a\}_{a\in S}$ 
of connected subsets of $M$ with pairwise disjoint interiors
such that $\varDelta=\bigcup_{a\in S}\varDelta_a$
and $f|_{\varDelta_a}=f_a$ for each $a\in S$.

 
 \item[(M1)] For each $a\in S$, $f_a$ 
is a $C^1$ diffeomorphism onto its image with bounded derivative.
 \item[(M2)] If $a,b\in S$ and $f\varDelta_a\cap \varDelta_b$ 
 has non-empty interior,
 then $f\varDelta_a\supset \varDelta_b$.
 
 \end{itemize}
 The family  $\{\varDelta_a\}_{a\in S}$  is called {\it a Markov partition} of $f$. 
 We say a Markov map $f$ is {\it fully branched}  $f\varDelta_a\supset\varDelta_b$ 
 holds for all $a\in S$, $b\in S$.
  We say $f$ 
 is {\it non-uniformly expanding} if
 $|f'|>1$ everywhere except for finitely many indices $(a,x_a)$, $a\in S$,
   for which $x_a\in\varDelta_a$ and $|f_a'x_a|=1$.
   Such indices $a$ are called {\it neutral indices},
   and $x_a$ is called a {\it neutral periodic point} if $f^nx_a=x_a$ holds for some $n\geq1$.
  All other indices are called {\it expanding}.
Let $\Omega$ denote the set of all neutral indices.
In the case $\Omega=\emptyset$, $f$ is called {\it uniformly expanding}.
 All neutral periodic points are topologically repelling.

      Condition (M2) determines a transition matrix over $S$.
       We will assume a strong connectivity of the
       associated directed graph, called {\it  finite irreducibility}.
    Moreover, we will  assume that $f$ has {\it uniform decay of cylinders}. We refer to Section \ref{CM} for the definitions of these two key properties.

       We aim to describe a fractal structure of the maximal invariant set
       $$J= \bigcap_{n=0}^\infty f^{-n}\varDelta.$$
      We assume $J\neq\emptyset$, and
     decompose $J$ into {\it level sets} defined as follows.
    Let $\mathcal{F}$ denote
    the set of $\mathbb R$-valued functions on $\varDelta$ with  suitable bounds on their distortion (see Section \ref{TF} for the precise definition).
    We say $f$ has {\it mild distortion} if $\log|f'|\in\mathcal F$.
      Let us introduce the following two infinite-dimensional vector spaces
      $$\mathcal{F}^{\mathbb N}=\{\boldsymbol\phi=(\phi_1,\phi_2,\ldots)\colon\phi_i\in\mathcal F\ \
      \text{$\forall i\geq1$}\}$$
      and
      $$\mathbb R^{\mathbb N}=\{\boldsymbol\alpha=(\alpha_1,\alpha_2,\ldots)\colon\alpha_i\in\mathbb R\ \
       \forall i\geq1\}.$$
      For $\boldsymbol\phi\in\mathcal F^{\mathbb N}$, $\boldsymbol\alpha\in\mathbb R^{\mathbb N}$ and an integer $k\geq1$,
     put  $\boldsymbol\phi_k=(\phi_1,\ldots,\phi_k)$, $\boldsymbol\alpha_k=(\alpha_1,\ldots,\alpha_k)$ and 
      $$\|\boldsymbol\alpha_k\|=\max_{1\leq i\leq k}|\alpha_i|.$$ 
    For an integer $n\geq1$ and $x\in\bigcap_{j=0}^{n-1}  f^{-j}\varDelta$, we denote the multi-Birkhoff sum by
     $$S_n\boldsymbol\phi_k(x)=\left(S_n\phi_1(x),\ldots,S_n\phi_k(x)\right),$$
      where
    $S_n\phi=\sum_{j=0}^{n-1}\phi\circ f^j$ for $\phi\colon \varDelta\to\mathbb R$. The level sets we consider are given by 
    $$B(\boldsymbol\phi,\boldsymbol\alpha)=\bigcap_{k=1}^\infty B_k(\boldsymbol\phi,\boldsymbol\alpha)$$ where we have set 
    $$B_k(\boldsymbol\phi,\boldsymbol\alpha)=\left\{x\in J\colon\lim_{n\to\infty}\left\|\frac{1}{n}
    S_n\boldsymbol\phi_k(x)-\boldsymbol\alpha_k\right\|=0\right\}.$$
Note that $B(\boldsymbol\phi,\boldsymbol\alpha)$ is the set of points in $J$ for which the Birkhoff average of each $\phi_i$ exists and is equal to $\alpha_i$.
We have the multifractal decomposition
\begin{equation*}\label{decomp}
J=\left(\bigcup_{\boldsymbol\alpha\in\mathbb R^{\mathbb N}} B(\boldsymbol\phi,\boldsymbol\alpha)\right)\cup B'(\boldsymbol\phi),\end{equation*}
where $B'(\boldsymbol\phi)$ denotes the set of points in $J$ for which the Birkhoff 
average of some $\phi_i$ does not exist. If $f|_J$ is transitive, then each non-empty level set
$B(\boldsymbol\phi,\boldsymbol\alpha)$ is a dense subset of $J$.
Our goal is to describe the {\it mixed multifractal spectrum of Birkhoff averages}
$$b_{\boldsymbol\phi}(\boldsymbol\alpha)={\rm dim}_HB(\boldsymbol\phi,\boldsymbol\alpha),\quad \boldsymbol\alpha \in \mathbb{R}^{\mathbb{N}},$$ where ${\rm dim}_H$ denotes the Hausdorff dimension on $M$. The metric on $M$ is the Euclidean metric if $M=[0,1]$, and
 the standard Riemannian metric if $M=S^1$.

All measures appearing in this paper are probabilities on appropriate Borel sigma-fields.
For each measure $\mu$ which is invariant under 
the map $f|_J\colon J\to J$,
 denote by $h(\mu)\in[0,\infty]$ the \emph{Kolmogorov-Sina{\u\i} entropy} of $\mu$ with respect to $f|_J$.
If $f$ is differentiable at $\mu$-almost every point, then define {\it the Lyapunov exponent} of $\mu$ 
 by $\chi(\mu)=\int\log|f'|d\mu\in[0,\infty]$.
Let $\mathcal M(f)$ denote the set of $f|_J$-invariant measures with finite Lyapunov exponent.
{\it The dimension} of a measure $\mu\in\mathcal M(f)$ is
the number $$\dim(\mu)=\begin{cases}
\displaystyle{\frac{h(\mu)}{\chi(\mu)}}&\ \text{ if }\chi(\mu)>0;\\
0&\ \text{ if }\chi(\mu)=0.\end{cases}$$  
A measure $\mu\in\mathcal M(f)$ is called {\it expanding} if  $\chi(\mu)>0$.
For an ergodic expanding measure $\mu$, 
$\dim(\mu)$ is equal to the supremum of the Hausdorff dimensions of sets
with full $\mu$-measure \cite{Hof95,HofRai92,Led81,MauUrb03}.

Our formula for $b_{\boldsymbol\phi}(\boldsymbol\alpha)$ is given in terms of the dimension of measures.
 Important quantities are
 $$\delta_0=\sup\{\dim(\mu)\colon\mu\in\mathcal M(f)\}$$
 and $$\beta_\infty=\inf
\left\{\beta\in\mathbb R\colon \sup\{F_\beta(\mu)\colon\mu\in\mathcal M(f)\}<\infty\right\}.$$
Here, $F_\beta\colon\mu\in\mathcal M(f)\mapsto
h(\mu)-\beta\chi(\mu)$ is
{\it a free energy with inverse temperature} $\beta$.
For $\boldsymbol\phi\in\mathcal F^{\mathbb N}$ and $k\geq1$,
we denote  $$\int\boldsymbol\phi_k d\mu=\left(\int\phi_1d\mu,\ldots,\int\phi_kd\mu\right),$$ where we will always assume that $\phi_i\in L^1(\mu)$ for every $1\leq i\leq k$, if this notation is in use.
For $\alpha\in[0,\infty]$ define
$$L(\alpha)=\left\{x\in J\colon\lim_{n\to\infty}\frac{1}{n}\log|(f^n)'x|=\alpha\right\}.$$
We are now in the position to state our main result. 
Let us say that $f$ is {\it saturated by expanding measures}, or simply {\it saturated} if
   \begin{equation}\label{dim}
\dim_HJ=\delta_0.\end{equation}

\begin{maintheorem}[Conditional variational formulas for mixed Birkhoff spectra]
Let $f\colon\varDelta\to M$
 be a finitely irreducible non-uniformly expanding Markov map which 
has mild distortion and uniform decay of cylinders.
   Let $\boldsymbol\phi\in\mathcal{F}^{\mathbb N}$ and $\boldsymbol\alpha\in\mathbb R^{\mathbb N}$.
  Then the following holds:
  \smallskip

 \noindent {\rm (a)}     We have $B(\boldsymbol\phi,\boldsymbol\alpha)
\neq\emptyset$  if and only if for any integer $k\geq1$
 and any $\epsilon>0$ there exists 
a measure
  $\mu\in\mathcal M(f)$  such that 
   $$\left\|\int\boldsymbol\phi_kd\mu-\boldsymbol\alpha_k \right\|<\epsilon.$$
 
  \noindent {\rm (b)} 
  Assume $f$ is saturated by expanding measures.
 If $B(\boldsymbol\phi,\boldsymbol
 \alpha)
\neq\emptyset$ then 
\begin{align*}b_{\boldsymbol\phi}(\boldsymbol\alpha)
&=\lim_{k\to\infty}\dim_HB_k(\boldsymbol\phi,\boldsymbol\alpha)\\
&=\lim_{k\to\infty}
   \lim_{\epsilon\to0}\sup\left\{\dim(\mu)\colon \mu\in\mathcal M(f),\
   \left\|\int\boldsymbol\phi_k d\mu-\boldsymbol\alpha_k\right\|<\epsilon\right\}.
\end{align*}
If moreover each $\phi_i$ is bounded, then 
\begin{align*}b_{\boldsymbol\phi}(\boldsymbol\alpha)
&=\lim_{k\to\infty}\lim_{\epsilon\to0}
   \sup\left\{\sup\left\{\dim(\mu)\colon \mu\in\mathcal M(f),\
   \left\|\int\boldsymbol\phi_k d\mu-\boldsymbol\alpha_k\right\|<\epsilon\right\},\beta_\infty\right\}.
\end{align*}
If moreover $f$ has a neutral periodic point, then
\begin{align*}
\delta_0=\dim_HL(0)=\lim_{\alpha\to0}\dim_HL(\alpha).\end{align*}
\end{maintheorem}

Note that the two limits  in the conditional variational formulas can be exchanged because of the monotonicity of the supremum in  $\epsilon$ and $k$. Further, we may replace $\mathcal M(f)$ by the subset  of measures with compact support.

Our Main Theorem allows us to investigate  a number of important multifractal spectra  which could not be studied before. As a first application we state new fractal-geometric  results for backward continued fractions in Section 1.2. 
In Section 2 we describe cusp winding spectra for hyperbolic surfaces of  finitely generated free Fuchsian groups with parabolic elements. 

Moreover, our Main Theorem extends the results of Fan et al. \cite[Theorem 1.1, Theorem 1.2]{FanJorLiaRam16}
 in which only uniformly expanding fully branched  interval maps were considered. Note that our  class $\mathcal F$ of observables is strictly larger than the one considered in \cite{FanJorLiaRam16}. Let us also remark that our proof differs significantly from that of   \cite{FanJorLiaRam16}. We refer to Section \ref{method} for further details.
 
Let us finally comment on the assumptions in our Main Theorem. The saturation assumption \eqref{dim}, which  seems very natural  for our task, has also been assumed in \cite{JJOP10}, and it follows from the assumptions in \cite{FanJorLiaRam16}. We state a sufficient condition in Section \ref{dist-cont} via inducing schemes. The uniform decay of cylinders clearly holds for uniformly expanding maps, and for non-uniformly expanding maps we have to state it as an assumption. A number of concrete examples has this property, see \cite[Section 8]{MauUrb03} and Section \ref{vdecay} for details.

 \subsection{Dimension spectra for  backward continued fraction expansions}
We describe our results in some number-theoretic problems.
Each irrational number $x\in(0,1)\setminus\mathbb Q$ has a unique 
expansion 
   \begin{equation}\label{expansion}x=1-\cfrac{1}{b_1(x)-\cfrac{1}{b_2(x)-\cfrac{1}{\ddots}}},\end{equation}
    where each digit $b_j(x)$ is an integer greater than or equal to $2$. 
     The digits in this expansion are generated by iterating the R\'enyi map \cite{Ren57}
     \begin{equation*}f\colon x\in[0,1)\mapsto\frac{1}{1-x}-\left\lfloor\frac{1}{1-x}\right\rfloor\in[0,1).\end{equation*}
        This means that for all $x\in(0,1)\setminus\mathbb Q$,
      $$b_j(x)=\left\lfloor\frac{1}{1-f^{j-1}x}\right\rfloor+1\quad \forall j\geq1.$$  
      The graph of the R\'enyi map can be obtained from that of the Gauss map 
      by reflecting the latter in the line $x=1/2$. For this reason, \eqref{expansion} is called the {\it
      Backward Continued Fraction} (BCF) expansion of the irrational number $x$.
       The R\'enyi map
    is a fully branched non-uniformly expanding Markov map having $x=0$ as a unique neutral fixed point.
    It satisfies all the assumptions in the Main Theorem (see Section 5 and the proof of Lemma \ref{lalpha}).

    
    The behavior of the arithmetic mean of the BCF digits is peculiar.
Aaronson \cite{Aar86} proved that the arithmetic mean convergences to $3$ in measure
as $n\to\infty$.
Aaronson and Nakada \cite{AarNak03} proved that
$$\liminf_{n\to\infty}\frac{1}{n}\sum_{j=1}^nb_j(x)=2\ \text{ and }\ 
\limsup_{n\to\infty}\frac{1}{n}\sum_{j=1}^nb_j(x)=\infty$$
for Lebesgue a.e. $x\in(0,1)\setminus\mathbb Q$.
Slightly modifying the proof of the Main Theorem we show that
some mixed Birkhoff spectra including the arithmetic mean 
of the BCF digits are completely flat.

\begin{theorem}[Completely flat mixed Birkhoff spectrum]\label{degenerate}
Let $f$ be the R\'enyi map. For any $\alpha\in[2,\infty]$
and any $\boldsymbol\phi=(\phi_i)_{i=1}^\infty\in\mathcal F^{\mathbb N}$ 
such that $\phi_i$ is bounded for every $i\geq1$, we have
  \[
 \dim_H \left\{
  \begin{tabular}{l}
  \vspace{5pt}
 $x\in (0,1)\setminus\mathbb Q\colon
 \displaystyle{\lim_{n\to\infty}\frac{1}{n}S_n\phi_i(x)=\phi_i(0)}\ \ \forall i\geq1$,\\
 \vspace{5pt}
 $\quad\quad\quad\quad\quad\quad\ \displaystyle{ \lim_{n\to\infty}\frac{1}{n}\log |(f^n)'x|=0,}$\\
    $\quad\quad\quad\quad\quad\quad\  \displaystyle{  \lim_{n\to\infty}\frac{1}{n}\sum_{j=1}^nb_j(x)=\alpha}$\\
  \end{tabular}\right\}=1.\]
\end{theorem}

 The following proposition is a special case of Theorem \ref{degenerate}. 
The Birkhoff spectrum 
of the arithmetic mean of the BCF digits is completely flat, in sharp contrast to 
that of the 
regular continued fraction expansion digits, which 
is a  strictly increasing real-analytic function \cite{IomJor}.
\begin{prop}[Completely flat Birkhoff spectrum]\label{degen}
For any $\alpha\in[2,\infty]$,
\[
 \dim_H \left\{x\in (0,1)\setminus\mathbb Q\colon 
 \lim_{n\to\infty}\frac{1}{n}(b_1(x)+\cdots+b_n(x))=\alpha
  \right\}=1.
\]\end{prop}

For each $n\geq2$ consider the set
   \begin{equation*}
   E(n)=\{x\in[0,1)\colon b_j(x)\leq n\ \ \forall j\geq1\}.    
   \end{equation*}
By a result of Urba\'nski (\cite[Theorem 4.3]{Urb96}), the Hausdorff dimension of
$E(n)$ coincides with the first zero $t_n\in[0,1]$ of the geometric pressure function associated with the Markov map generated by  $n-1$ branches of the R\'enyi map.
Since the R\'enyi map has no invariant Borel probability  measure that is absolutely continuous with respect to the Lebesgue measure, we have
 $t_n<1$ by a  result of Ledrappier \cite{Led81b}. Since 
 the R\'enyi map is saturated by expanding measures, and since the support of any invariant measure with compact support is contained in $E(n)$, for some $n\ge 2$, our Main Theorem proves the following, 
  answering 
a question of Pollicott (private communication in the Fall Program of Low-Dimensional Dynamics, Shanghai, October 2019).
\begin{prop}[Limit of dimension of sets with bounded BCF digits]\label{limit-d}
We have
$$\displaystyle{\lim_{n\to\infty}}\dim_HE(n)=1.$$
\end{prop}

It follows that
the set of irrational numbers in $(0,1)$ with bounded BCF digits has Hausdorff dimension $1$. 
For the regular continued fraction expansion, Jarn\'ik \cite{Jar32} proved that the set of 
irrational numbers in $(0,1)$ whose digits do not exceed $n$ has Hausdorff dimension $1-O(1/n)$.
This was later significantly improved by Hensley \cite{Hen92} to $1-O(1/n^{2}).$


To state our next main result let $f\colon\varDelta\to M$ be a Markov map 
with an infinite Markov partition $\{\varDelta_i\}_{i=1}^\infty$. We say that $\boldsymbol\alpha \in {\mathbb R}^{\mathbb N}$ is a  {\it  frequency vector} if
$\alpha_i\geq0$ holds for every $i\ge 1$ and $\sum_{i=1}^{\infty} \alpha_i\leq1$.
For each frequency vector $\boldsymbol\alpha$ we introduce the {\it  Besicovitch-Eggleston set} with frequency $\boldsymbol\alpha$ given by 

  $$BE(\boldsymbol\alpha)=\left\{x\in J\colon \lim_{n\to\infty}\frac{1}{n}
     \#\{0\leq j\leq n-1\colon f^jx\in \varDelta_i\}=\alpha_i
     \ \ \forall i\ge 1\right\}.$$ 
     For the R\'enyi map we have 
 $$BE(\boldsymbol\alpha)=\left\{x\in(0,1)\setminus\mathbb Q\colon \lim_{n\to\infty}\frac{1}{n}\#\{1\leq j\leq n\colon b_j(x)=i\}=\alpha_{i-1}\ \ \forall i\geq2\right\}.$$     
 
 \begin{theorem}[Dimension of Besicovitch-Eggleston sets]\label{dimthm}
 Let $f\colon[0,1)\to [0,1)$ be the R\'enyi map. For every frequency vector $\boldsymbol\alpha$ we have 
\[ \dim_HBE(\boldsymbol\alpha)
   =\lim_{k\to\infty}\lim_{\epsilon\to0}
   \max\left\{\sup\left\{\dim(\mu)\colon \mu\in\mathcal M(f),\
   \max_{1\leq i\leq k}\left|\mu(\varDelta_i) -\alpha_i\right|<\epsilon\right\}, \frac{1}{2}\right\}.
 \]
 Moreover, if $\sum_{i=1}^\infty\alpha_i<1$ then we have $$\dim_HBE(\boldsymbol\alpha)=\frac{1}{2}.$$
    \end{theorem}


     For fully branched uniformly expanding
Markov maps with infinite Markov partitions,
  the conditional variational formulas for the Hausdorff dimension of Besicovitch-Eggleston sets in Theorem \ref{dimthm} were established in \cite{FanJorLiaRam16,FLM10}. See \cite{FLMW10} for analogous results on the countable Markov shift. 
The second assertion of Theorem \ref{dimthm} extends the result 
  \cite[Theorem 1.2]{FanJorLiaRam16} on uniformly expanding fully branched Markov maps.
This result is closely related to a result of Good \cite{Goo41} stating that  set of numbers with regular continued fraction digits tending to infinity has  Hausdorff dimension equal to $1/2$. For a related result on cutting sequences of geodesics on hyperbolic surfaces with cusps we refer to Proposition \ref{prop:multicuspfreq} in Section 2 and the remark thereafter.
Statements analogous to Theorem \ref{dimthm} hold for general non-uniformly expanding Markov maps with infinite Markov partitions.

%
     
     \subsection{Method of proof and outline of the paper}\label{method}
      The proof of the Main Theorem  differs significantly from that of \cite[Theorem 1.1, Theorem 1.2]{FanJorLiaRam16} in which the uniform expansion of the Markov map is essential.  
As stated in the  Main Theorem(b), 
the set of points with zero pointwise Lyapunov exponent has a large Hausdorff dimension,
 and can therefore intersect many level sets.
 It is convenient to restrict to level sets with positive Lyapunov exponents,
 but many level sets will be unaccounted for due to this restriction
 (see e.g. \cite[Theorem 2]{JJOP10}).
 To establish our Main Theorem for the complete multifractal spectra, we take the set with zero Lyapunov exponent into a direct consideration.

What we will actually prove is the conditional variational formula  
for any level set whose Hausdorff dimension does not exceed $\delta_0$.
The saturation  assumption ensures that the formulas are valid for all level sets.
For finitely generated iterated function systems with neutral fixed points, 
essentially the same assumption  was made in \cite[Theorem 2]{JJOP10} 
in order to describe Birkhoff spectra of single observables.

In Section 2 we will prove related results  for the multi-cusp winding process for the geodesic flow on hyperbolic surfaces with cusps. The rest of this paper consists of four sections. Section 3 contains preliminary results on non-uniformly expanding one-dimensional 
         Markov maps.
In Section 4 we prove the Main Theorem.
For the upper bound of the Hausdorff dimension of the level sets,  we modify 
 Bowen's covering argument  \cite{Bow79}, and utilize 
 the thermodynamic formalism
 for finite and countable Markov shifts
 \cite{Bow75,MauUrb01,Rue78} applied to induced systems. 
 We will work with
 infinitely many induced systems which altogether exhaust the whole dynamics. 
For the lower bound,
we construct subsets of the level sets by a Moran-like geometric construction \cite{Pes97}, and relate the Hausdorff dimension of the level sets to the dimension of expanding measures with prescribed expected values.

In Section 5 we provide  sufficient conditions for two of the assumptions in the Main Theorem: the uniform decay of cylinders,
and  the saturation
by expanding measures.
In Section 6 we show Theorem \ref{degenerate} and Theorem \ref{dimthm}
for the R\'enyi map.

\section{Cusp winding spectra for some hyperbolic surfaces}\label{BowSer}

In this section we apply our main results to cusp winding spectra for the geodesic flow on  hyperbolic surfaces modeled by a finitely generated free Fuchsian group with parabolic elements. The underlying 
non-uniformly expanding Markov map is 
the Bowen-Series map associated with  the action of a Fuchsian group on the boundary of hyperbolic space (\cite{BowSer79}). This induces a uniformly expanding Markov map with infinitely many branches which is not fully branched. This fact has also been used in \cite{KesStr04}.
Other cusp winding spectra were treated in \cite{JaeKesMun16,Mun12}. 

\subsection{Bowen-Series maps for free Fuchsian groups}
We denote by $(\mathbb D,d)$ the Poincar\'e disk model of the two-dimensional hyperbolic space, where $\mathbb D\subset \mathbb R^2$ is the open unit disk around the origin, and $\mathbb S^1=\partial\mathbb D$ its boundary. We denote by $$\Lambda(G)=\overline{\bigcup_{g\in G}g(0)}\setminus\bigcup_{g\in G}g(0) \subset \mathbb{S}^1$$ the \emph{limit set} of a Fuchsian group $G$ (see \cite{Nic89} and also \cite[Section 1]{BowSer79}  for further details),
where the closure is taken with respect to the Euclidean topology on $\mathbb R^2$.

We begin by recalling
the definition of the Bowen-Series map from \cite{BowSer79}. 
Let $R\subset\mathbb{D}$ denote a Dirichlet fundamental domain for
a finitely generated free Fuchsian group $G$ with parabolic elements.
By Poincar\'e's polyhedron theorem, each of the finitely
many sides $s$ of $R$ gives rise to a side-pairing transformation
$g_{s}\in G$, and the set 
\[
G_{0}=\left\{ g_{s}\colon s\text{ }\text{side of }R\right\} \subset G
\]
is a symmetric set of generators of $G$. Moreover,  each vertex $p_{j}$ of $R$ ($j=1,\dots,k)$
satisfies $p_{j}\in\mathbb{S}^{1}$ and that $p_{j}$ is a parabolic
fixed point of some $g\in G$. We will assume for simplicity that
for each $p_{j}$ there exists $\gamma_j\in G_0$ such that $\gamma_j(p_j)=p_j$.  We define 
\[
\Gamma_0=\{ \gamma_j^{\pm 1} \in G_0 \colon 1\le j\le k\}\ \text{ and }\ H_{0}=G_{0}\setminus\Gamma_{0}.
\]
We will always assume that $G$ is non-elementary, which in our setting means that $\# G_0\ge 4$. Each side $s$ of $R$ is contained  in the isometric
circle of $g_s$, where we recall that the isometric circle $C_{g}$ of $g\in G$ is given by 
\[
C_g=\left\{ z\in\mathbb{R}^{2}\colon |g'(z)|=1\right\}.
\]
Here,  $|g'(z)|$ denotes the norm of the derivative of $g$ at $z$ with respect
to Euclidean metric on $\mathbb{R}^{2}$. We then define for $g\in G_0$,
\[
\varDelta_{g}=\mathbb{S}^{1}\cap \left\{ z\in\mathbb{R}^{2}\colon |g'(z)|\ge1\right\}.
\]
Let $\varDelta=\bigcup_{g\in G_{0}}\varDelta_{g}$ and define the Bowen-Series
    map $f:\varDelta\to\mathbb S^1$ by 
\[
f|_{\varDelta_{g}}=g|_{\varDelta_{g}},\quad g\in G_{0}.
\]
That $f$ is well defined follows from the fact that, for distinct $g,h\in G_{0}$, 
we have
$\varDelta_{g}\cap \varDelta_{h}=\emptyset$ unless $h=g^{-1}\in\Gamma_{0}$. 
In this case, we have $\varDelta_{g}\cap \varDelta_{g^{-1}}=\left\{ p_{j}\right\} $ 
where $p_{j}$ is the parabolic fixed point of $g$, and thus $f(p_{j})=p_{j}$. 

\begin{figure}
\begin{center}
\includegraphics[height=4.5cm,width=4cm]{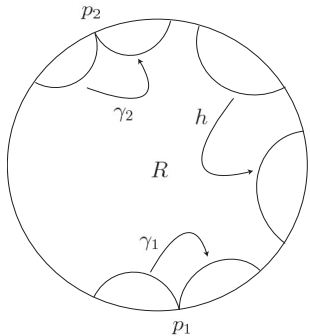}
\caption{A Dirichlet fundamental domain $R$ of a finitely generated free Fuchsian group generated by $G_0=H_0 \cup \Gamma_0$ with $H_0=\{h,h^{-1} \}$ and $\Gamma_0=\{\gamma_1,\gamma_1^{-1},\gamma_2,\gamma_2^{-1}  \}$.}
\end{center}
\end{figure}

Note that $f$ is a Markov map with Markov partition $\{\varDelta_{g}\}_{g\in G_{0}}$.
It is not difficult to verify that $f$ is a finitely irreducible  non-uniformly expanding
Markov map satisfying Renyi's condition and (M3).
Hence, $f$ has uniform decay of cylinders by Lemma \ref{uniform}.
For the maximal $f$-invariant set $J\subset\mathbb{S}^{1}$ we have
\[
J=\Lambda(G).
\]
There exists a subset
 of the set  $\{(\omega_n)_{n=0}^\infty\in G_0^{\mathbb N}
 \colon \omega_{n}\omega_{n+1}\neq1\in G\ \forall n\geq0\}$ 
 of symbolic sequences which is in one-to-one correspondence with $\Lambda(G)$
 (\cite{BowSer79,Ser86}). This correspondence is given by the coding map associated with the Markov map $f$ as defined in Section \ref{CM}.
We denote by $\Lambda_c(G)\subset \Lambda(G)$ the conical limit set of $G$ (see \cite{Nic89} for the definition). By a result of Beardon and Maskit \cite{BeaMask74} we have 
\begin{equation} 
\Lambda_{c}(G)=\Lambda(G)\setminus\bigcup_{g\in G}\bigcup_{j=1}^{k}g^{-1}(p_{j}).\label{eq:beardonmaskit}
\end{equation}

\subsection{Multi-cusp winding process}
 For $x\in \Lambda_c(G)$ it follows from \eqref{eq:beardonmaskit} that the corresponding symbolic sequence $(\omega_n)_{n=0}^\infty\in G_0^\mathbb N$ does not eventually become constant to some element of $\Gamma_0$. For  $x\in \Lambda_c(G)$ we can therefore decompose  its symbolic sequence $\omega$ into a  sequence of blocks $(B_{i}(x))_{i\ge1}$ as in \cite{JaeKesMun16}.
Each hyperbolic generator in $\omega$ forms a block of length one. For parabolic
generators in $\omega$ we build maximal blocks of consecutive appearances of the
same parabolic generator. We have either  $B_i(x)=h$ for some $h\in H_0$, or $B_i(x)=\gamma^n$ for some $\gamma\in \Gamma_0$ and $n\geq1$.

Motivated by \cite{JaeKesMun16}, we introduce
the \emph{multi-cusp winding process} $(a_{i,j})_{i\geq1,1\leq j\leq k}$, a family of functions on $\Lambda_{c}(G)$ given by 
\[
a_{i,j}(x)=\begin{cases}
n-1 & \text{if }B_{i}(x)=\gamma_{j}^{\pm n},\quad n\ge1\text{ }\text{ }\\
0 & \text{otherwise.}
\end{cases}
\]
This process has an interpretation in terms of the geodesic
flow on the associated hyperbolic surface  $\mathbb D/G$  (\cite{Ser86}).
Namely, for a fixed initial value $o\in \mathbb D/G$, the set of oriented geodesics $\gamma:[0,\infty)\rightarrow \mathbb D/G$ which satisfy $\gamma(0)=o$ and which return to some compact set of $\mathbb D/G $ infinitely often, can be identified with $\Lambda_c(G)$. This identification is obtained by lifting $\gamma$ to a geodesic $\tilde{\gamma}:[0,\infty)\rightarrow \mathbb{D}$ with $\tilde{\gamma}(0)\in R$, and identifying $\gamma$ with the endpoint $x:=\lim_{t \to \infty} \tilde{\gamma}(t)\in \mathbb{S}^1$. The  symbolic sequence $\omega$
is then the  cutting sequence which records the sides of $R$ crossed by $\gamma(t)$ as    $t\rightarrow \infty$ (\cite[Section 2]{{Ser86}}). A block $B_i(x)=\gamma_j^n$ of length $n\ge2$ of some
parabolic generator $\gamma_{j}\in \Gamma_0$ then means that the geodesic $\gamma$  spirals
$n-1$ times around the cusp associated with $\gamma_{j}$. Here, the $i$th block refers to the $i$th return of $\gamma$ to a certain compact subset of $\mathbb D/G$.

\subsection{Mixed Birkhoff spectra for multi-cusp winding process}
In order to apply our theory for mixed Birkhoff spectra to the multi-cusp
winding process, we introduce the induced Markov map $\tilde{f}:K\rightarrow\mathbb{S}^{1}$
given by 
\[
\tilde{f}x=f^{|B_{1}(x)|}x,
\]
where  $|B_1(x)|$ denotes the word length of the block $B_1(x)$ and  $K\subset\mathbb{S}^{1}$ is the  countable union of pairwise disjoint
partition elements
\[
K=\bigcup_{\omega\in F}\varDelta_{\omega}.
\]
Here,  $F$ is the  countable set of reduced words in $\bigcup_{n=2}^\infty G_{0}^n$  given
by
\[
F=\bigcup_{n=1}^\infty\left\{ \gamma^{n}g\colon
\gamma\in\Gamma_{0},g\in G_{0}\setminus\left\{ \gamma^{\pm1}\right\} \right\} \cup\left\{ hg\colon h\in H_{0},\,\,g\in G_{0}\setminus\left\{ h^{-1}\right\} \right\} .
\]
By \eqref{eq:beardonmaskit},
the maximal $\tilde{f}$-invariant set $\tilde{J}$ satisfies
\[
\tilde{J}=\Lambda_{c}(G).
\]
Let 
\[
D_{k}=\left\{ \boldsymbol{\alpha}=(\alpha_{1},\dots,\alpha_{k})\in\mathbb{R}^{k}\colon\alpha_{j}\ge0\quad 1\le \forall j\le k\right\}.
\]
For $\boldsymbol{\alpha}\in D_{k}$ we define 
\[
B(\boldsymbol{\alpha})=\left\{ x\in\Lambda_{c}(G)\colon\lim_{n\rightarrow\infty}\frac{1}{n}(
a_{1,j}(x)+\cdots +a_{n,j}(x))=\alpha_{j}\quad 1\le \forall j\le k\right\} .
\]

\begin{prop}[Birkhoff spectrum of multi-cusp windings]\label{multicusp} For every $\boldsymbol{\alpha}\in D_{k}$
we have $B(\boldsymbol{\alpha})\neq \emptyset$ and 
\[
\dim_{H}B(\boldsymbol{\alpha})=\lim_{\epsilon\to0}\sup\left\{ \dim(\mu)\colon\mu\in\mathcal{M}(\tilde{f}),\ \left|\int a_{1,j}\,d\mu-\alpha_{j}\right|<\epsilon\quad1\le \forall j\le k\right\} .
\]
\end{prop}

\begin{proof}
For $i\ge1$ and $1\le j\le k$ we have 
$a_{i,j}=a_{1,j}\circ\tilde{f}^{i-1}.$
Hence, 
$B(\boldsymbol{\alpha})$ is the set of points in $\tilde J$
for which the Birkhoff average of $a_{1,j}$ for $\tilde f$
is equal to $\alpha_j$ for each $1\leq j\leq k$.
Our main task is
to verify that $\tilde{f}$ is a finitely irreducible, non-uniformly 
expanding Markov map with  a Markov partition $\{\varDelta_{\omega}\}_{\omega\in F}$
which has mild distortion and uniform decay of cylinders, and satisfies \eqref{dim}.
That $\tilde{f}$ is finitely irreducible follows from the fact that
 $f|_J$ is transitive with the finite Markov partition $\{\varDelta_{\omega}\}_{g\in G_0}$. Since $f$ has uniform decay of cylinders, so does $\tilde{f}$. 
Since $B_1$ is constant on each element of the Markov partition,
 $\tilde f$ satisfies \eqref{dim} by Proposition \ref{verify}(b), and 
 $\tilde f$ has mild distortion by Proposition \ref{verify}(c).
Hence, the desired conditional variational formula follows from  Main Theorem(b) for every non-empty level set $B(\boldsymbol{\alpha})$.

For each $n\ge 1$ and $1\le j \le k$ we consider the $\tilde{f}$-invariant probability measure on the periodic orbit $(\gamma_j^n g \gamma_j^n g \cdots)$ of $\tilde f$, where $\gamma_j \in \Gamma_0$ and  $g\in G_0\setminus\{\gamma_j^{\pm1}\}$. By taking suitable convex combinations of these measures and applying Main Theorem(a) we can show that $B(\boldsymbol{\alpha})\neq\emptyset$  for every $\boldsymbol{\alpha}\in D_{k}$.
\end{proof}
\begin{remark}
In the case $k=1$, our conditional variational formula for the cusp
winding spectrum is analogous to the Birkhoff spectrum of the arithmetic
mean of the regular continued fraction digits considered in \cite[Corollary 6.6]{IomJor}.
Note that in there the dynamical system is given by the Gauss map.
\end{remark}


In order to investigate the frequency of multi-cusp windings, we define for $i\ge0$ and $1\le j\le k$
the sets 
\[
A_{i,j}=\left\{ x\in\Lambda_{c}(G)\colon a_{1,j}(x)=i\right\} .
\]
Denote
by $D$ the set of frequency vectors $\boldsymbol{\alpha}=(\alpha_{i,j})\in\mathbb{R}^{\mathbb{N}\times\left\{ 1,\dots,k\right\} }$.
Then  for every $\boldsymbol{\alpha}\in D$ we have 
\[
BE(\boldsymbol{\alpha})
=\left\{ x\in\Lambda_{c}(G)\colon\lim_{n\rightarrow\infty}\frac{1}{n}\#\left\{ 1\le\ell\le n:a_{\ell,j}(x)=i\right\} =\alpha_{i,j},\quad\forall i\ge0,\ 1\le \forall j\le k\right\}.
\]
Hence, the number  $\alpha_{i,j}$ prescribes the asymptotic frequency of precisely
$i$ consecutive windings around the $j$th cusp.

\begin{prop}[Frequency of multi-cusp windings]\label{prop:multicuspfreq}For every  $\boldsymbol{\alpha}\in D$ 
we have $BE(\boldsymbol{\alpha})\neq\emptyset$  and  

 \begin{align*}  &\dim_H BE(\boldsymbol\alpha)=\\
   &\lim_{k\to\infty}\lim_{\epsilon\to0}
   \max\left\{\sup\left\{\dim(\mu)\colon \mu\in\mathcal M(\tilde{f}),\
   \sup_{i\geq0,\,\,1\leq j\leq k}\left|\mu(A_{i,j}) -\alpha_{i,j}\right|<\epsilon\right\},\frac{1}{2}\right\}.
   \end{align*}
Moreover, for any  $\boldsymbol \alpha \in D$ satisfying $\sum_{i=1}^\infty \sum_{j=1}^k \alpha_{i,j}<1$ we have 
\[
\dim_H BE(\boldsymbol \alpha)=\frac{1}{2}.
\]

\end{prop}

\begin{proof}
Recall that  $\tilde{f}$ is a finitely irreducible, non-uniformly
expanding Markov map with a Markov partition $(\varDelta_{\omega})_{\omega\in F}$
which has mild distortion and uniform decay of cylinders, and which
satisfies \eqref{dim}. Further, each $A_{i,j}$ is a finite union of elements of $(\varDelta_{\omega})_{\omega\in F}$. 
By \cite[Lemma 2.4]{JaeKesMun16}  we have $\beta_{\infty}=1/2$. Finally, that $BE(\boldsymbol{\alpha})\neq\emptyset$ for every  $\boldsymbol{\alpha}\in D$ can be shown as in the proof of Proposition \ref{multicusp}.
\end{proof}

\begin{remark}
In the case $k=1$, the last assertion in Proposition \ref{prop:multicuspfreq} is related to \cite[Theorem 2]{Mun12}
where it is shown that, for a certain Jarn\'ik set $\mathcal{J}\subset\Lambda_{c}(G)$,
we have $\dim_{H}\mathcal{J}=1/2$. Namely, by the definition of $\mathcal{J}$, 
we have $\lim_{n\rightarrow\infty}a_{n,1}(x)=\infty$ for every $x\in\mathcal{J}$, which implies  
$\mathcal{J}\subset BE(\boldsymbol0)$.
\end{remark}

 \section{Preliminaries on non-uniformly expanding Markov maps}
Throughout this section, let $f\colon \varDelta\to M$ be a non-uniformly expanding Markov map
with a Markov partition $\{\varDelta_a\}_{a\in S}$.
In Section \ref{CM} we introduce a symbolic coding via the Markov partition,
as well as the necessary definitions and  conditions appearing in our  main results.
In Section \ref{full} we introduce a subset $J^*$ of $J$ of full Hausdorff dimension which detects  weak expansion.
In Section \ref{TF} we define the set $\mathcal F$ of admissible observables for our  main results.
In Section \ref{ap} we state two lemmas  on approximations of invariant measures, which will be  frequently used later on. 

 \subsection{Symbolic coding}\label{CM}
 
 Condition (M2) determines a transition matrix $(T_{ab})$ over the countable alphabet $S$  by
 $T_{ab}=1$ if $f\varDelta_a\supset\varDelta_b$ and $T_{ab}=0$ otherwise.
  {\it A word of length $n\geq1$}
   is an $n$-string of elements of $S$.
A word  $\omega_0\cdots \omega_{n-1}$ 
  of length $n$ is {\it admissible} if $n=1$, or else $n\geq2$ and $T_{\omega_{j}\omega_{j+1}}=1$ holds for every $0\leq j\leq n-1$.
 Denote by $E^n$ the set of $n$-admissible words and put $E^*=\bigcup_{n=1}^\infty E^n$.
 For two words $\omega=\omega_0\cdots \omega_{m-1}$ and $\eta=\eta_0\cdots  \eta_{n-1}$,
 denote by $\omega\eta$ the concatenated word
$\omega_0\cdots \omega_{m-1}\eta_0\cdots \eta_{n-1}$.
This notation extends in an  obvious way to concatenations of an arbitrary finite number of words.
 For convenience, put   $E^0=\{ \emptyset\}$,  $|\emptyset|=0$, and  $\eta\emptyset=\eta=\emptyset\eta$ 
 for  every  $\eta\in E^*$.
We say $f$ is {\it finitely irreducible} if there exists a finite set $\Lambda\subset E^*$ such that
for any $\omega,\eta\in E^*$ there exists $\lambda\in \Lambda$ such that $\omega\lambda \eta\in E^*$. 
If $f$ is fully branched, i.e., $T_{ab}=1$ for all $a,b\in S$, then it is finitely irreducible with $\Lambda=\emptyset$.

Let $\mathbb N$ denote the set of non-negative integers.
The transition matrix determines a topological Markov shift
$$X=\{x=(x_n)_{n=0}^\infty\colon x_n\in S,\ T_{x_nx_{n+1}}=1 \ \ \forall n\in \mathbb N\},$$
endowed with the usual symbolic metric.
The left shift acting on $X$ is denoted by $\sigma$:
$(\sigma x)_n=x_{n+1}$ for every $n\in\mathbb N$. 
 For each $\omega=\omega_0\cdots \omega_{n-1}\in E^n$, put $|\omega|=n$ and
 define 
 the {\it $n$-cylinder}
$$\varDelta_\omega=\bigcap_{j=0}^{n-1} f^{-j}\varDelta_{\omega_j},$$
and put
  $$[\omega]=[\omega_0,\ldots,\omega_{n-1}]=\{x\in  X\colon x_{j}=\omega_j
 \ \ 0\leq\forall j\leq n-1\}.$$
 The length of the cylinder $\varDelta_\omega$ is denoted by $|\varDelta_\omega|$.
We say $f$ has {\it uniform decay of cylinders}
if
$$\displaystyle{\lim_{n\to\infty}}\sup_{\omega\in E^n}|\varDelta_\omega|=0.$$
  If $f$ has uniform decay of cylinders,
  then for any $(x_n)_{n=0}^\infty\in X$ the set
 $\bigcap_{n=0}^\infty\overline{f^{-n}\varDelta_{x_n}}$ is a singleton.
The coding map $\pi\colon X\to M$ given by
  $$\pi((x_n)_{n=0}^\infty)\in \bigcap_{n=0}^\infty\overline{f^{-n}\varDelta_{x_n}}$$ is well-defined and satisfies $J\subset \pi(X)$.
  It is measurable, 
  one-to-one except on the countable set $\bigcup_{n=0}^\infty f^{-n}(\bigcup_{a\in S}\partial\varDelta_a)$.
  Since the set $\{x\in X\colon\text{$\pi$ is one-to-one at
 $\pi(x)$}\}$ is $\sigma$-invariant and
 $f\circ\pi=\pi\circ \sigma$ holds on this set, 
for any non-atomic $f$-invariant measure $\mu$ on $\pi(X)$, there 
exists a $\sigma$-invariant measure $\mu'$ 
such that $\mu=\mu'\circ\pi^{-1}$ and the entropy of $\mu'$ with respect to $\sigma$ is equal to $h(\mu)$.
\subsection{Weak expansion on a set of full dimension}\label{full}
For a neutral index $a\in\Omega$ and an integer $n\geq1$, let
$\varDelta^n(a)$ denote the union of all $n$-cylinders containing $x_a$
(there are at least one and at most two such cylinders, see the remark below).

\begin{remark}\label{rem3} Since the elements of the Markov partition may intersect each other at their boundary points, it may happen that $a,b\in \Omega$, $a\neq b$ and $x_a=x_b$.
  \end{remark}

Fix an integer $n_0\geq1$ such that the following holds:

\begin{itemize}
\item[(i)] If $a,b\in\Omega$ and $x_a\neq x_b$, then  $\varDelta^{n_0}(a)\cap\varDelta^{n_0}(b)=\emptyset.$

\item[(ii)]
For each $a\in\Omega$, there exists at most one $b\in\Omega$
such that  $f(\varDelta^{n_0}(a))\cap \varDelta^{n_0}(b)$ has non-empty interior.

\end{itemize}
Put
$$V=\overline{\bigcup_{a\in\Omega}\varDelta^{n_0}(a)}.$$
This is a closed neighborhood of the set $\{x_a\colon a\in\Omega\}.$
Define
$$J^*=J\cap\bigcap_{n=0}^\infty\bigcup_{k=n}^\infty f^{-k}(M\setminus V).$$
Points in $J^*$ can have zero pointwise Lyapunov exponent. Nevertheless, the following weak expansion estimate holds.
\begin{lemma}\label{w-exp}
Let $x\in J^*$. For any $n'\geq1$ there exists an integer $n\geq n'$
such that for every $\omega\in E^n$ such that
$x\in\varDelta_\omega$, we have 
$\inf_{\varDelta_\omega}|(f^n)'|\ge \inf_{\varDelta\setminus V}|f'|>1.$
\end{lemma}
\begin{proof}
Let $n>\max\{n_0,n'\}$ be such that
$f^{n-n_0} x \notin V$.
For every $\omega\in E^n$ such that $x\in\varDelta_\omega$, $f^{n-n_0}\varDelta_{\omega}$ is a $n_0$-cylinder.
Since cylinders of the same word lengths have disjoint interiors,
$f^{n-n_0}\varDelta_{\omega}\subset\overline{\varDelta\setminus V}$ holds.
For every $y\in \varDelta_\omega$ we have
$$|(f^n)'y|=|(f^{n_0-1})'(f^{n-n_0+1}y)||f'(f^{n-n_0}y)||(f^{n-n_0})'y|\geq
|f'(f^{n-n_0}y)|\geq\inf_{\varDelta\setminus V}|f'|>1,$$
which yields the claim.
\end{proof}
We restrict ourselves to the set $J^*$, and
use Lemma \ref{w-exp} to avoid the effect of neutral periodic points for the  upper estimate of the Hausdorff dimension of the level sets in our main theorem. 
That this restriction carries the full  Hausdorff dimension follows from the next lemma.
\begin{lemma}\label{B}
If $f$ has uniform decay of cylinders, then 
$J\setminus J^*$ is a countable set.
\end{lemma}
\begin{proof}
Condition (M2) implies that if $f(\varDelta^{n_0}(a))\cap \varDelta^{n_0}(b)$ has non-empty interior then $f(\varDelta^{n_0}(a))\supset \varDelta^{n_0}(b)$.
From this and the choice of $n_0$, it follows that
 for each $a\in\Omega$, all points in 
 $\overline{\varDelta^{n_0}(a)}\cap\bigcap_{n=1}^\infty f^{-n}V$, except possibly a countable number of points which eventually enter the boundary of some element of the Markov partition, 
 share the same coding sequence.
By the uniform decay of cylinders, 
there exists at most one point in $M$ for each given coding sequence. Since $\Omega$ is finite it thus follows that  $\bigcap_{n=0}^\infty f^{-n}V$ is countable. Therefore, 
   $J\setminus J^*=J\cap\bigcup_{k=0}^\infty f^{-k}\left(\bigcap_{n=0}^\infty f^{-n}V\right)$ is a  countable set.
 \end{proof}

 \subsection{Class of observables}\label{TF}
  For a function $\phi\colon \varDelta\to\mathbb R$  
and an integer $n\geq1$ put
$$D_n(\phi)=\sup_{\omega\in E^n}\sup_{x,y\in\varDelta_\omega}S_n\phi(x)-S_n\phi(y),$$
 and define
$$\mathcal{F}=\{\phi\colon \varDelta\to\mathbb R\colon D_1(\phi)<\infty\text{
and }D_n(\phi)=o(n) \}.$$
The first condition is non-trivial only if $\phi$ is unbounded.
We will only consider observables which belong to $\mathcal F$. If $f$ has uniform decay of cylinders,
$\mathcal F$ contains the set of bounded uniformly continuous functions.





We say $f$ satisfies {\it R\'enyi's condition} if each branch $f_a$ $(a\in S)$ is $C^2$ and satisfies
$$\sup_{a\in S}\sup_{\varDelta_a}\frac{|f_a''|}{|f_a'|^2}<\infty.$$

 \begin{lemma}\label{mild}
If $f$ has uniform decay of cylinders and satisfies R\'enyi's condition,
then it has mild distortion, namely $\log|f'|\in\mathcal F$.

\end{lemma}

\begin{proof}
For every $a\in S$ and all $x,y\in\varDelta_a$ we have
$$\log\frac{|f_a'x|}{|f_a'y|}\leq\sup_{\varDelta_a} \frac{|f_a''|}{|f_a'|^2}|fx-fy|\leq \sup_{\varDelta_a}\frac{|f_a''|}{|f_a'|^2}.$$
Therefore $D_1(\log|f'|)<\infty$.
Iterating this argument, for every $n\geq2$, every $\omega\in E^n$ and all $x,y\in \varDelta_\omega$
we obtain
$$\log\frac{|(f^n)'x|}{|(f^n)'y|}\leq \sup_{a\in S}\sup_{\varDelta_a}\frac{|f_a''|}{|f_a'|^2}\sum_{j=1}^{n}|f^j\varDelta_\omega|
\leq\sup_{a\in S}\sup_{\varDelta_a}\frac{|f_a''|}{|f_a'|^2} \left(\sum_{j=1}^{n-1}\sup_{\omega\in E^{n-j}}|\varDelta_{\omega}|+1\right).$$
The uniform decay of cylinders implies that the last sum is $o(n)$, and therefore
$f$ has mild distortion.
\end{proof}


\subsection{Approximations of expanding measures}\label{ap}

For a subset $F$ of
$E^*$ we use the following notation: $$[F]=\bigcup_{\omega\in F} [\omega],\quad
|F|=\sup_{\omega\in F}|\omega|,\quad
\varDelta F=\bigcup_{w\in F}\varDelta_\omega.$$

 The lemma below
 allows us to approximate an ergodic measure
 with a finite collection of cylinders under the assumption of finite irreducibility.
 For $a,b\in S$ and an integer $n\geq1$, let $E^n(a,b)$ denote the set of elements of $E^n$
whose first symbol is $a$ and the last one is $b$. 

   \begin{lemma}\label{katok}
    Let $f$ be finitely irreducible and have uniform decay of cylinders.
  Let $\boldsymbol\phi\in\mathcal{F}^{\mathbb N}$,
  $k\geq1$ and let
$\mu\in\mathcal M(f)$ be an ergodic expanding measure such that $\phi_i\in L^1(\mu)$ for every $1\leq i\leq k$.
For any $\epsilon>0$ and any integer $n\geq1$,
there exist $\ell\geq n$, $a,b\in S$ and
a finite subset $F^{\ell}(a,b)$ of $E^{\ell}(a,b)$ such that 
\begin{equation*}
\left|\frac{1}{\ell}\log\#F^{\ell}(a,b)- h(\mu)\right|<\epsilon\ \text{ 
and }\
\sup_{\varDelta F^{\ell}(a,b)}\left\|\frac{1}{\ell}S_{\ell}\boldsymbol\phi_k-\int\boldsymbol\phi_k d\mu\right\|<\epsilon.\end{equation*}
\end{lemma}
\begin{proof}
By virtue of the uniform decay of cylinders, $f|_J$ is semi-conjugate to a countable Markov shift
via the Markov partition.
Then the proof of Lemma \ref{katok} is carried out on a symbolic level as in
\cite[Lemma 2.3]{Tak}.
\end{proof}
Since the finite irreducibility implies the transitivity,
the proof of \cite[Main Theorem]{Tak} works verbatim to show the next lemma
approximating non-ergodic measures with ergodic ones in a particular sense.

\begin{lemma}\label{nonergodic}
 Let $f$ be finitely irreducible and have uniform decay of cylinders.
Let $\boldsymbol\phi\in\mathcal{F}^{\mathbb N}$, $k\geq1$ and let
$\mu\in\mathcal M(f)$ be a non-ergodic expanding measure such that $\phi_i\in L^1(\mu)$ for every $1\leq i\leq k$.
For any $\epsilon>0$ there exists an ergodic expanding measure $\mu'\in\mathcal M(f)$ which satisfies
$$|h(\mu)-h(\mu')|<\epsilon\ \text{ and }\ 
\left\|\int\boldsymbol\phi_k d\mu-\int\boldsymbol\phi_k d\mu'\right\|<\epsilon.$$
\end{lemma}

\section{On the proofs of the Main results}
This section is dedicated to proofs of the Main Theorem and related results, including Theorem \ref{dimthm}.
In Section \ref{bsec} we prove an upper bound (Proposition \ref{up}) leading to the formula in Main Theorem(b).
In Section \ref{keylow} we prove a lower bound (Proposition \ref{lowd}), and
 we put these together in Section \ref{end} to complete the proof of the Main Theorem.
 In Section \ref{pthmb} we prove Theorem \ref{dimthm}.

\subsection{Upper bound on dimension}\label{bsec}

Let $f\colon\varDelta\to M$, 
 $\boldsymbol\phi\in\mathcal{F}^{\mathbb N}$, 
 $\boldsymbol\alpha\in\mathbb R^{\mathbb N}$
 be as in the Main Theorem.
For each integer $k\geq1$ and $\epsilon>0$ put
$$ d_k(\epsilon)=\sup \left\{\dim(\mu)\colon\mu\in\mathcal M(f),\
\left\|\int\boldsymbol\phi_k d\mu-\boldsymbol\alpha_k\right\|<\epsilon\right\}+\epsilon,$$
and set $$ d_k=\displaystyle{\lim_{\epsilon\to0}}\ d_k(\epsilon).$$
The next proposition will be used to obtain upper bounds on the Hausdorff dimension of level sets.
\begin{prop}\label{up}
Let $f\colon\varDelta\to M$ be a non-uniformly expanding Markov map which is finitely irreducible, has mild distortion and  uniform decay of cylinders.
Let $\boldsymbol\phi\in\mathcal{F}^{\mathbb N}$ and $\boldsymbol\alpha\in\mathbb R^{\mathbb N}$, and
let $k\geq1$ be such that $B_k(\boldsymbol\phi,\boldsymbol\alpha)\neq\emptyset$.
Assume \begin{equation}\label{advantage}d_k<\delta_0.\end{equation}
Then the following holds:
\begin{itemize}
\item[(a)] If $\sup_\varDelta\|\boldsymbol\phi_k\|=\infty$,
then
$${\rm dim}_HB_k(\boldsymbol\phi,\boldsymbol\alpha)\leq d_k;$$
\item[(b)]  If $\sup_\varDelta\|\boldsymbol\phi_k\|<\infty$,
then
$${\rm dim}_HB_k(\boldsymbol\phi,\boldsymbol\alpha)\leq \max\{d_k,\beta_\infty\}.$$
\end{itemize}
\end{prop}


 \begin{proof}
 The proofs of (a) and (b) involve the same set of ideas that is outlined as follows.
For a subset $G$ of $E^*$ and $\beta\in\mathbb R$, put
$$|G|^{\beta}=\sum_{\omega\in G} |\varDelta_\omega|^{\beta}.$$
For each $a\in S$ and sufficiently small $\epsilon>0$, we will construct a family $\{G_n\}_n$ of subsets of $E^*$ with increasing
uniform word lengths such that
 the corresponding family of cylinders covers  
$B_k(\boldsymbol\phi,\boldsymbol\alpha)\cap J^*\cap\varDelta_a$,
and satisfies $\lim_{n\to\infty}|G_n|^{d_k(\epsilon)}=0$.
We would like to apply the thermodynamic formalism 
for the induced system associated with each $G_n$, to
 obtain an upper bound of $|G_n|^{d_k(\epsilon)}$ in terms of the Lyapunov exponent of a
certain expanding measure. 
The problem is that such a measure depends on $n$,
and can have arbitrary small Lyapunov exponent 
due to the existence of neutral periodic points.
Hence, we fix below 
an expanding measure with sufficiently large dimension,
 and use it to construct a family of expanding measures whose Lyapunov exponents are bounded away from zero.
\begin{lemma}\label{ab}
There exists an expanding measure 
 $\mu^*\in\mathcal M(f)$ such that
 $\phi_i\in L^1(\mu^*)$ for every $1\leq i\leq k$ and
\begin{equation*}
 d_k<\dim(\mu^*).\end{equation*}
 \end{lemma}
\begin{proof}
By virtue of \eqref{advantage}
and Lemma \ref{nonergodic}, 
there exists an ergodic expanding measure 
 $\mu$ such that $d_k<\dim(\mu).$
If $\phi_i\in L^1(\mu)$ for every $1\leq i\leq k$,
then put $\mu^*=\mu$. Otherwise, 
we approximate $\mu$ with a finite number of cylinders in the sense of Lemma \ref{katok}, and glue them together using the finite irreducibility 
to construct another expanding measure $\mu^*$ whose
entropy and Lyapunov exponent are arbitrarily close to those of $\mu$. Since the support of $\mu^*$ is
contained in the union of finitely many cylinders, 
$\phi_i\in L^1(\mu^*)$ for every $1\leq i\leq k$.
 \end{proof}
 Fix an expanding measure $\mu^*$ as in Lemma \ref{ab}.
 Put
\begin{equation}\label{eq20}
A_k=\max\left\{\left\|\int\boldsymbol\phi_k d\mu^*\right\|,2\|\boldsymbol\alpha_k\|,1\right\}>0.\end{equation}
Fix $\epsilon_0\in(0,7A_k)$ such that
 \begin{equation}\label{eq--1}
 F_{d_k(\epsilon_0)}(\mu^*)\geq0.\end{equation}
Put
$$e_k(\epsilon)=\max\{d_k(\epsilon),\beta_\infty+\epsilon\}.
$$

    \begin{lemma}\label{horse}
For any $\epsilon\in(0,\min\{\epsilon_0,4/3\})$ there exists $n(\epsilon)\geq1$ such that if
$n\geq n(\epsilon)$, $a\in S$ and $G$ is a non-empty subset of $ E^{n+1}(a,a)$ 
such that for every $\omega\in G$,
 \begin{equation}\label{horse-eq}
 \inf_{\varDelta_\omega}|(f^n)'|>1
 \ \text{ and } \
 \sup_{\varDelta_\omega} \left\|\frac{1}{n}S_{n}\boldsymbol\phi_k-\boldsymbol\alpha_k\right\|<\frac{\epsilon}{2},\end{equation}
then the following holds:
\begin{itemize}
\item[(a)] If $G$ is a finite set, 
then 
 $$ |G|^{d_k(\epsilon)}\leq \exp\left(-\frac{\epsilon^2 \chi(\mu^*)n}{8A_k}\right).$$
\item[(b)]  If $\sup_\varDelta\|\boldsymbol\phi_k\|<\infty$, 
then 
 $$
 |G|^{e_k(\epsilon)}\leq \exp\left(-\frac{\epsilon^2 \chi(\mu^*)n}{8A_k}\right).$$
 \end{itemize}
\end{lemma}

We finish the proof of Proposition \ref{up} assuming the conclusion of Lemma \ref{horse}. 
Let $\epsilon\in(0,\min\{\epsilon_0,4/3\})$ and $n\geq n(\epsilon)$. 
We assume $n$ is large enough so that
\begin{equation}\label{eq-L1}
D_n(\log|f'|)<\epsilon^3 n\ \text{ and }\ D_{n}(\phi_i)<\frac{\epsilon n}{4}\ \  1\leq \forall i\leq k.\end{equation}
Let $a\in S$. Put
$$H^n=\left\{
  \begin{tabular}{l}
 $\omega\in \bigcup_{b\in S}E^{n}(a,b)\colon
 \displaystyle{\inf_{\varDelta_\omega}|(f^n)'|>1,}
    \quad
    \displaystyle{  
    \inf_{\varDelta_\omega}\left\|\frac{1}{n}S_n\boldsymbol\phi_k(z)-\boldsymbol\alpha_k\right\|<\frac{\epsilon}{4}}$
  \end{tabular}\right\}.$$

\begin{proof}[Proof of Proposition \ref{up}$(a)$]
The next lemma asserts that 
one can cover the level set with countably many subsets with `finite ranges'.
\begin{lemma}\label{f-union}
If $\sup_\varDelta\|\boldsymbol\phi_k\|=\infty,$ then
$B_k(\boldsymbol\phi,\boldsymbol\alpha)$
is decomposed into a countable union of 
subsets $B_{k,p}(\boldsymbol\phi,\boldsymbol\alpha)$ ($p=1,2,\ldots$)
such that
for each $p\geq1$ and $n\geq1$ there 
exists a finite union $C_{n,p}$ of $1$-cylinders 
such that
$f^n(B_{k,p}(\boldsymbol\phi,
\boldsymbol\alpha))\subset C_{n,p}$.
\end{lemma}
\begin{proof}
Let $1\leq i\leq k$ be such that $\sup_\varDelta|\phi_i|=\infty.$ For each integer $t\geq1$ put
$$B_k(\boldsymbol\phi,\boldsymbol\alpha,t)=\left\{x\in
B_k(\boldsymbol\phi,\boldsymbol\alpha)\colon\left|\frac{1}{n}S_n\phi_i(x)-\alpha_i\right|\leq1
\ \ \forall n\geq t\right\}.$$
Note that
$B_k(\boldsymbol\phi,\boldsymbol\alpha)=\bigcup_{t=1}^\infty B_k(\boldsymbol\phi,\boldsymbol\alpha,t).$
The unboundedness of $\phi_i$ implies that
if $n\geq t$ then $f^n(B_k(\boldsymbol\phi,\boldsymbol\alpha,t))$ is contained
in a finite union of $1$-cylinders, the number of which may depend on $n$.
The set $I=\bigcup_{t=1}^\infty\{t\}\times E^t$ 
is a countably infinite set. Write $I=\{a_p\}_{p=1}^\infty$ and
define
$B_{k,p}(\boldsymbol\phi,\boldsymbol\alpha)=B_{k}(\boldsymbol\phi,\boldsymbol\alpha,t)\cap\varDelta_{\omega}$
where $(t,\omega)=a_p.$
Then $\{B_{k,p}(\boldsymbol\phi,\boldsymbol\alpha)\}_{p=1}^\infty$ has the desired properties.
\end{proof}
By Lemma \ref{w-exp} and Lemma \ref{f-union}, for each $p\geq1$ with
$B_{k,p}(\boldsymbol\phi,\boldsymbol\alpha)\neq\emptyset$
there exists a finite set $H_{p}^n\subset H^n$ such that
 $$B_{k,p}(\boldsymbol\phi,\boldsymbol\alpha)\cap J^*\cap\varDelta_a\subset
 \bigcup_{n=N}^\infty \varDelta H^{n}_p.$$
For each $\omega\in H_p^{n}$ fix $\lambda(\omega)\in\Lambda$ and  $z(\omega)\in \varDelta_\omega$
satisfying 
\begin{equation}\label{eq-L2}
\omega\lambda(\omega) a\in E^*\ \text{ and }\ \left\|\frac{1}{n}S_n\boldsymbol\phi_k(z(\omega))-\boldsymbol\alpha_k\right\|<\frac{\epsilon}{4}.\end{equation}
Using \eqref{eq-L1} and \eqref{eq-L2} we have
\begin{align*}
\sup_{\varDelta_{\omega\lambda(\omega)}}\left\|S_{n+|\lambda(\omega)|}\boldsymbol\phi_k-(n+|\lambda(\omega)|)\boldsymbol\alpha_k\right\|\leq&
\sup_{\varDelta_{\omega\lambda(\omega)}}\left\|S_{n}\boldsymbol\phi_k-S_{n}\boldsymbol\phi_k(z(\omega))\right\|\\
&+\left\|S_{n}\boldsymbol\phi_k(z(\omega))-n\boldsymbol\alpha_k \right\|\\
&+\sup_{\varDelta_{\omega\lambda(\omega)}}\left\|S_{|\lambda(\omega)|}\boldsymbol\phi_k\circ f^n-|\lambda(\omega)|\boldsymbol\alpha_k\right\|\\
<&\frac{\epsilon n}{2}+\sup_{\varDelta_{\lambda(\omega)}}\left\|S_{|\lambda(\omega)|}\boldsymbol\phi_k\right\|+
|\lambda(\omega)|\cdot\left\|\boldsymbol\alpha_k\right\|\\
<& \epsilon(n+|\lambda(\omega)|),
\end{align*}
where the last inequality holds for sufficiently large $n$
since $\Lambda$ is a finite set by the finite irreducibility.
For $0\leq q\leq|\Lambda|$ put 
 $$ H_p^n(q)=\{\omega\in H_p^{n}\colon |\lambda(\omega)|=q\}.$$
 Pick $q_0\in\{0,\ldots,|\Lambda|\}$ which maximizes the quantity
 $|\{\omega\lambda(\omega) a\colon \omega\in  H_p^n(\cdot)\}|^{d_k(\epsilon)}.$
 Put $$G_n=\{\omega\lambda(\omega) a\colon \omega\in  H_p^n(q_0)\}.$$ 
  Put $\rho=\min_{\lambda\in\Lambda}|\varDelta_{\lambda a}|>0$.
  For each $\omega\in H_p^n$ we have
    $$\frac{|\varDelta_{\omega\lambda(\omega) a}|}{|\varDelta_\omega|}\geq e^{-\epsilon^3 n}\frac{|f^n\varDelta_{\omega\lambda(\omega) a}|}
    {|f^n\varDelta_\omega|}
  \geq e^{-\epsilon^3 n}|\varDelta_{\lambda(\omega) a}|\geq e^{-2\epsilon^3 n}\rho,$$
  where the first inequality is from \eqref{eq-L1}.
 Hence
\begin{align*}\label{upeq0}
|H_p^{n}|^{d_k(\epsilon)}
&\leq C_{k,\epsilon,n}
|\left\{\omega\lambda(\omega)a\colon\omega\in H_p^{n}\right\}|^{d_k(\epsilon)}\\
&=C_{k,\epsilon,n}\sum_{q=0}^{|\Lambda|}
|\{\omega\lambda(\omega) a\colon \omega\in  H_p^n(q)\}|^{d_k(\epsilon)}\\
&\leq  C_{k,\epsilon,n}(|\Lambda|+1)
|G_n|^{d_k(\epsilon)},\end{align*}
where $C_{k,\epsilon,n}=(\rho^{-1} e^{2\epsilon^3 n})^{d_k(\epsilon)}$. 
Lemma \ref{horse}(a) applied to the finite set $G_n$ gives
\begin{align*}
\sum_{n=N}^\infty |H_p^{n}|^{d_k(\epsilon)}
\leq  \rho^{-d_k(\epsilon)} (|\Lambda|+1) \sum_{n=N}^\infty   \exp\left(\left(2\epsilon^3 d_k(\epsilon) -\frac{\epsilon^2\chi(\mu^*)}{8A_k}\right)n\right).\end{align*}
The last sum is summable for $\epsilon$ small enough.
From the uniform decay of cylinders,
the Hausdorff $d_k(\epsilon)$-measure of the set
$B_{k,p}(\boldsymbol\phi,\boldsymbol\alpha)\cap J^*\cap\varDelta_a$ is finite,
and so its Hausdorff dimension does not exceed $d_k$.
Since $p\geq1$ is arbitrary, we obtain
$\dim_H(B_{k}(\boldsymbol\phi,\boldsymbol\alpha)\cap J^*\cap\varDelta_a)\leq d_k$
as required.
\end{proof}

\noindent{\it Proof of Proposition \ref{up}$(b)$.}
Lemma \ref{w-exp} implies that 
for each integer $N\geq1$,
 $$B_k(\boldsymbol\phi,\boldsymbol\alpha)\cap J^*\cap\varDelta_a\subset\bigcup_{n=N}^\infty\varDelta H^n.
 $$
  The rest of the proof is a
 repetition of the previous argument. We make use of Lemma \ref{horse}(b) instead of Lemma \ref{horse}(a), and replace   $B_{k,p}(\boldsymbol\phi,\boldsymbol\alpha)$ by $B_k(\boldsymbol\phi,\boldsymbol\alpha)$, $H^n_p$ by $H^n$, and  $d_k(\epsilon)$ by $e_k(\epsilon)$. 
 \end{proof}

\begin{proof} [Proof of Lemma \ref{horse}]
Let $\epsilon\in(0,\min\left\{\epsilon_0,4/3\right\})$ and 
 $a\in S$. Let $n\geq 1$ and let $G$ be the non-empty subset of $E^{n+1}(a,a)$ in the statement of Lemma \ref{horse}.
 Put
  $\tilde X=\bigcap_{m=0}^\infty(\sigma^n)^{-m}[G]\subset X$,
  $\tilde\sigma=\sigma^{n}|_{\tilde X}$ and 
   $K=\pi\tilde X$, where $\pi$ is the coding map defined in Section \ref{CM}.
  Then $\tilde X$ is topologically conjugate to the full shift on $\#G$-symbols.
 The induced map
 $f^n|_K\colon K\to K$ 
 is topologically semi-conjugate to this full shift, namely $\pi|_{\tilde X}$ is one-to-one except on 
 countable number of points where it is two-to-one, and the following diagram commutes:
  \[
  \begin{CD}
     \tilde X @>{\tilde\sigma}>> \tilde X \\
  @V{\pi|_{\tilde X}}VV    @VV{\pi|_{\tilde X}}V \\
     K   @>{f^n}|_{K}>>  K.
  \end{CD}\]
 Define the induced potential $\tilde\psi\colon \tilde X\to\mathbb R$ given by
 $$\tilde\psi=-\log|(f^n)'\circ\pi|.$$
Fix $z\in K$.
For $\beta\in\mathbb R$ we have
  \begin{align*}\sum_{x\in \tilde\sigma^{-m}z}\exp\left(\beta\sum_{j=0}^{m-1}\tilde\psi(\tilde\sigma^jx)\right)
&\geq\left(\inf_{z'\in K}\sum_{x\in \tilde\sigma^{-1}z'}e^{\beta\tilde\psi(x)}\right)^m\geq\left(e^{-\beta D_n(\log|f'|)}
|G|^{\beta}\right)^m.\end{align*}
Taking logarithms, dividing by $m$ and then letting $m\to\infty$, we have 
$$\liminf_{m\to\infty}\frac{1}{m}\log\sum_{x\in \tilde\sigma^{-m}z}\exp\left(
\beta\sum_{j=0}^{m-1}\tilde\psi(\tilde\sigma^jx)\right)
\geq\log |G|^{\beta}-\beta D_{n}(\log|f'|).$$
That the  lower-limit is actually a limit  follows from sub-additivity. We show that this limit is equal to the pressure $P(\beta\tilde\psi)$ of the potential
$\beta\tilde\psi$ with respect to $\tilde \sigma$ as defined in \cite{MauUrb01}.
For every $m\geq1$ we have
$$D_m(\log|(f^n)'|_{K}|)\leq D_{mn}(\log|f'|)=o(mn).$$ Therefore, the Markov map $f^n|_{K}$ with  Markov partition $\{\varDelta_\omega\}_{\omega\in G}$
has mild distortion. 
By a slight modification of the proof of \cite[Theorem 1.2]{MauUrb01} we obtain 
 \begin{equation}\label{press}P(\beta\tilde\psi) =
 \lim_{m\to\infty}\frac{1}{m}\log\sum_{x\in \tilde\sigma^{-m}z}\exp\left(\beta\sum_{j=0}^{m-1}\tilde
 \psi(\tilde\sigma^jx)\right).\end{equation}
 Moreover, the pressure satisfies  the variational principle 
 \begin{equation}\label{vp}P(\beta\tilde\psi)=\sup_{\tilde\mu\in \mathcal M(\tilde\sigma)}\left\{
 \tilde h(\tilde\mu)+\beta\int\tilde\psi d\tilde\mu\colon\int\tilde\psi d\tilde\mu>-\infty  \right\},\end{equation}
where $\mathcal M(\tilde\sigma)$ denotes the space of $\tilde\sigma$-invariant measures
 endowed with the weak*-topology,
 and $\tilde h(\tilde\mu)$ the entropy of $\tilde\mu\in \mathcal M(\tilde\sigma)$
 with respect to $\tilde\sigma$.


The rest of the proof consists of two parts.
We first prove Lemma \ref{horse}(a). Assume $G$ is a finite set. For each $m\geq1$
and $x\in \tilde\sigma^{-m}z$, we have
$\#(\tilde\sigma^{-m}z)=(\#G)^m$ and
$\left|\sum_{j=0}^{m-1}\tilde
 \psi(\tilde\sigma^jx)\right|\leq m\sup_{\varDelta G}\log|(f^n)'|$.
 Hence for every $\beta\in\mathbb R$,
 $$P(\beta\tilde\psi)\leq\sup_{m\geq1}\frac{1}{m}\log\sum_{x\in \tilde\sigma^{-m}z}\exp\left(\beta\sum_{j=0}^{m-1}\tilde
 \psi(\tilde\sigma^jx)\right)\leq \log\#G+\sup_{\varDelta G}|(f^n)'|^{\beta}<\infty.$$
Substituting $\beta=d_k(\epsilon)$ into \eqref{vp} gives
$$\infty> P(d_k(\epsilon)\tilde\psi)\geq\log  |G|^{d_k(\epsilon)}-C_{k,\epsilon,n}',$$
where $C_{k,\epsilon,n}'=d_k(\epsilon) D_{n}(\log|f'|)$.
Pick $\tilde\nu\in\mathcal M(\tilde\sigma)$ such that $\int\tilde\psi d\tilde\nu>-\infty$, and
\begin{equation*}
\tilde h(\tilde\nu)+d_k(\epsilon)\int\tilde\psi d\tilde\nu\geq
\log |G|^{d_k(\epsilon)}-\epsilon-C_{k,\epsilon,n}'.\end{equation*}
The measure $\nu=(1/n)\sum_{j=0}^{n-1}f^j_*(\pi|_{\tilde X})_*\tilde\nu$
is in $\mathcal M(f)$.
Since $\sup_{[G]}\tilde\psi<0$ 
from the first inequality in \eqref{horse-eq},
$\int\tilde\psi d\tilde\nu<0$ and so
 $\chi(\nu)>0$ holds.
The second inequality in \eqref{horse-eq} gives 
\begin{equation}\label{eq21}
\left\|\int\boldsymbol\phi_kd\nu-\boldsymbol\alpha_k\right\|\leq
\frac{\epsilon}{2}.\end{equation}
Hence, by the definition of $d_k(\epsilon)$ we have $$F_{d_k(\epsilon)}(\nu)\leq-\epsilon\chi(\nu)<0.$$
This yields
\begin{equation}
\begin{split}\label{eq--4}
0>F_{d_k(\epsilon)}(\nu)n
&=\tilde h(\tilde\nu)+d_k(\epsilon)\int\tilde\psi d\tilde\nu\\
&\geq
\log |G|^{d_k(\epsilon)}-\epsilon-C_{k,\epsilon,n}'.\end{split}
\end{equation}
The Lyapunov exponent of $\nu$ may become arbitrarily close to $0$
as $n$ increases. 
To circumvent this problem, we introduce another expanding measure
$$\xi=\left(1-\frac{\epsilon}{7A_k}\right)\nu+\frac{\epsilon}{7A_k}\mu^*,$$
which is in $\mathcal M(f)$ and satisfies
\begin{align*}
    \left\|\int\boldsymbol\phi_k d\xi-\boldsymbol\alpha_k\right\|\leq&
    \left(1-\frac{\epsilon}{7A_k}\right)\left\|\int\boldsymbol\phi_k d\nu-\boldsymbol\alpha_k\right\|+\frac{\epsilon}{7A_k}\left\|\int\boldsymbol\phi_kd\nu\right\|\\
    &+\frac{\epsilon}{7A_k}\left\|\int\boldsymbol\phi_kd\mu^*\right\|\\
    \leq&\frac{\epsilon}{2}+\frac{\epsilon}{7A_k}\left\|\int\boldsymbol\phi_kd\nu\right\|+\frac{\epsilon}{3}\quad \text{by \eqref{eq20} and \eqref{eq21}}\\
    <&\frac{5\epsilon}{6}+\frac{\epsilon}{7A_k}\left(\|\boldsymbol\alpha_k\|+
\frac{\epsilon}{2}\right)\quad \text{by \eqref{eq21}}\\
<&\frac{5\epsilon}{6}+\frac{\epsilon}{14}+\frac{\epsilon^2}{14}\quad \text{by \eqref{eq20}}\\
<&\epsilon.
\end{align*}
The last inequality holds if $\epsilon<4/3$. Therefore,
$d_k(\epsilon)\geq0$ and
$$F_{d_k(\epsilon)}(\xi)\leq-\epsilon\chi(\xi)\leq-\frac{\epsilon^2}{7A_k}\chi(\mu^*).$$
On the other hand, from \eqref{eq--1} and \eqref{eq--4},
\begin{align*}
F_{d_k(\epsilon)}(\xi)n=&
\left(1-\frac{\epsilon}{7A_k}\right)F_{d_k(\epsilon)}
(\nu)n+\frac{\epsilon}{7A_k}
F_{d_k(\epsilon)}(\mu^*)n\\
\geq&
\log  |G|^{d_k(\epsilon)}-\epsilon-C_{k,\epsilon,n}',\end{align*}
and therefore
\begin{align*}
\log  
|G|^{d_k(\epsilon)}&\leq
F_{d_k(\epsilon)}(\xi)n+\epsilon+C_{k,\epsilon,n}'\\
&\leq-\frac{\epsilon^2}{7A_k}\chi(\mu^*) n+\epsilon+C_{k,\epsilon,n}'\\
&\leq-\frac{\epsilon^2\chi(\mu^*) n}{8A_k}.\end{align*}
The last inequality holds for $n$ large enough. This completes the proof of Lemma \ref{horse}(a).

We are  left to prove Lemma \ref{horse}(b). If $G$ is a finite set,
then the desired inequality is a consequence of Lemma \ref{horse}(a)
and $e_k(\epsilon)\geq d_k(\epsilon)$. If $G$ is a countably infinite set,
then the variational principle does not preclude the case
where both numbers in \eqref{vp} are $\infty$.
We claim this is not the case for $\beta=e_k(\epsilon)$.
Indeed,
 for any measure $\tilde\mu\in\mathcal M(\tilde\sigma)$ with $\int\tilde\psi d\tilde\mu>-\infty$,
the measure $\mu=(1/n)\sum_{j=0}^{n-1}f^j_*(\pi|_{\tilde X})_*\tilde\mu$ is in $\mathcal M(f)$ and satisfies
   \begin{align*}
   F_{e_k(\epsilon)}(\mu)n=
   \tilde h(\tilde\mu)+e_k(\epsilon)\int\tilde\psi d\tilde\mu,
\end{align*}
by Abramov-Kac's formula (\cite[Theorem 2.3]{PesSen08} \cite[Theorem 5.1]{Zwe05}).
This implies $$P(e_k(\epsilon)\tilde\psi)\leq n\cdot\sup\{F_{e_k(\epsilon)}(\mu)\colon\mu\in\mathcal M(f)\}.$$
Since 
 $e_k(\epsilon)>\beta_\infty$, the
  right-hand side is finite.
We repeat the same argument as before, replacing $d_k(\epsilon)$ by $e_k(\epsilon)$.
This completes the proof of Lemma \ref{horse}.
 \end{proof}

 \subsection{Lower bound on dimension}\label{keylow}
The next proposition will be used to obtain lower bounds on the Hausdorff dimension of level sets.

\begin{prop}\label{lowd}
Let $f\colon\varDelta\to M$ be a non-uniformly expanding Markov map
which is finitely irreducible and has uniform decay of cylinders.
Let $\boldsymbol\phi\in\mathcal{F}^{\mathbb N}$,
$\boldsymbol\alpha\in\mathbb R^{\mathbb N}$ and
let $\{\mu_i\}_{i=1}^\infty$ be a sequence of ergodic expanding measures
such that 
$$\lim_{i\to\infty}\left\|\int\boldsymbol\phi_i d\mu_i-\boldsymbol\alpha_i\right\|=0.$$
Then   $B(\boldsymbol\phi,\boldsymbol\alpha)\neq\emptyset$ and
\begin{equation*}
b_{\boldsymbol\phi}(\boldsymbol\alpha)\geq \limsup_{i\to\infty}\dim(\mu_i).\end{equation*}
\end{prop}



\begin{proof}[Proof of Proposition \ref{lowd}]
The proof consists of two steps.
We first approximate each measure $\mu_i$ with a finite collection of cylinders in the sense of Lemma \ref{katok},
and use the finite irreducibility to glue orbits together to construct a Borel subset $\Gamma$ 
such that $\dim_H\Gamma\geq\limsup_{i\to\infty}\dim(\mu_i)$ in Step 1.
 Next we show in Step 2 that
 $\Gamma$ is contained in $B(\boldsymbol\phi,\boldsymbol\alpha)$.
\medskip

\noindent{\it Step 1: Construction of a Borel set $\Gamma$.}
We begin by claiming that we may  assume $h(\mu_i)>0$ for every $i\geq1$ without any loss of generality.
To see this, fix an expanding measure $\nu$ and define a sequence $\{\nu_i\}_{i=1}^\infty$ of expanding measures as follows:
if $h(\mu_i)>0$ then $\nu_i=\mu_i$. If $h(\mu_i)=0$, then take $\delta_i\in(0,1)$ 
 and define $\nu_i=(1-\delta_i)\mu_i+\delta_i\nu$ so that
 $\lim_{i\to\infty}$$\left\|\int\boldsymbol\phi_i d\nu_i-\boldsymbol\alpha_i\right\|=0$
 and $\limsup_{i\to\infty}\dim(\mu_i)=\limsup_{i\to\infty}\dim(\nu_i).$
 The claim follows from Lemma \ref{nonergodic}.

Choose a subsequence of $\{\mu_i\}_{i=1}^\infty$ so that $\{\dim(\mu_i)\}_{i=1}^\infty$ 
converges to the maximal possible limit. 
Let $\Lambda$ be the finite subset of $E^*$ 
given by the finite irreducibility of $f$.
 By taking a further subsequence and relabeling the indices if necessary, we may assume
 the following holds for every $i\geq1$:
 \begin{equation}\label{equa1}
 h(\mu_i)>\frac{1}{i};\end{equation} 
 \begin{equation}\label{equa2}
 \left\|\int\boldsymbol\phi_i d\mu_i-\boldsymbol\alpha_i\right\|<\frac{1}{i}.\end{equation}
  Put 
 $$c_i=\sup_{\lambda\in\Lambda}\sup_{\varDelta_{\lambda}}\left\|S_{|\lambda|}\boldsymbol\phi_i\right\|.$$
 Since $\boldsymbol\phi\in\mathcal F^{\mathbb N}$  and  $\Lambda$ is finite,
$c_i<\infty$ holds for every $i\ge 1$. Clearly, $c_i$ is monotone increasing in $i$.
 In view of Lemma \ref{katok}, fix
an integer $\ell_i\geq1$, 
symbols $a_i,b_i\in S$ and a finite subset
$F^{\ell_i}=F^{\ell_i}(a_i,b_i)$ of $E^{\ell_i}(a_i,b_i)$ 
for which the following holds for every $i\geq1$:
\begin{equation}\label{loweq1}
\ell_i\geq i;\end{equation}
\begin{equation}\label{eq-ell}
\frac{\ell_i}{i}\geq\sup_{\lambda\in\Lambda}\sup_{\varDelta_\lambda}\log|(f^{|\lambda|})'|;
\end{equation}
 \begin{equation}\label{important2}
ic_i+i|\Lambda|\cdot\left\|\boldsymbol\alpha_i\right\|\leq\frac{\ell_i}{i};
\end{equation}
\begin{equation}\label{qe1}\left|\frac{1}{\ell_i}\log \#F^{\ell_i}-h(\mu_i)\right|\leq\frac{1}{i};\end{equation}
\begin{equation}\label{qe3}\sup_{\varDelta F^{\ell_i}}\left\|\frac{1}{\ell_i}S_{\ell_i}\boldsymbol\phi_i
-\int \boldsymbol\phi_i d\mu_i\right\|\leq\frac{1}{i};
\end{equation}
\begin{equation}\label{qe2}\sup_{\varDelta F^{\ell_i}}\left|\frac{1}{\ell_i}\log|(f^{\ell_i})'|-\chi(\mu_i)\right|\leq\frac{1}{i}.
\end{equation}
For each $(b_i,a_i)$
fix $\lambda_{b_ia_i}\in\Lambda$ such that $b_i\lambda_{b_ia_i}a_i\in E^*$,
and for each $(b_{i},a_{i+1})$ fix
$\lambda_{b_ia_{i+1}}\in\Lambda$ such that $b_i\lambda_{b_ia_{i+1}}a_{i+1}\in E^*$.
For each integer $t\geq0$ we define a finite subset $F^{\ell_i,t}$ of $E^*$ as follows.
Put $F^{\ell_i,0}=E^0$ and $F^{\ell_i,1}=F^{\ell_i}$.
For $t\geq2$, $F^{\ell_i,t}$ is the set of elements of $E^*$ of the form
$$\omega_1 \lambda_{b_ia_i} \omega_2\lambda_{b_ia_i}\cdots  \omega_{t-1}\lambda_{b_ia_i}\omega_{t}$$
with $\omega_1,\ldots,\omega_{t}\in F^{\ell_i}$.
Clearly we have
\begin{equation}\label{sls}
\#F^{\ell_i,t}= \left(\#F^{\ell_i}\right)^{t}. 
 \end{equation}
 
 Let $\{t_i\}_{i=1}^\infty$ be a sequence of positive integers to be determined later.
For integers $N\geq2$ and $s\geq0$,
 denote by $E^{t_1,\ldots,t_{N-1},s}$ the set of elements of $E^*$
of the form $$\eta_1\lambda_{b_1a_{2}} \eta_2\lambda_{b_2a_3}\cdots \eta_{N-1}\lambda_{b_{N-1}a_N} \eta_N $$ with
$\eta_i\in F^{\ell_i,t_i}$ for $i=1,\ldots,N-1$
and $\eta_N\in F^{\ell_N,s}$.
Clearly, all words in
   $E^{t_1,\ldots,t_{N-1},s}$ have the same length and
\begin{equation}\label{sls'}
\begin{split}
 \#E^{t_1,\ldots,t_{N-1},s}=\left(\prod_{i=1}^{N-1}\#F^{\ell_i,t_i}\right)\max\{\#F^{\ell_N,s},1\}.
 \end{split}\end{equation}
 For each integer $T\geq t_1$ there exists a unique $N=N(T)$ and $s=s(T)$ such that 
\begin{equation}\label{k}
T=t_1+\cdots+t_{N-1}+s,\ 0\leq s\leq t_N-1.\end{equation}
Set
$$\Gamma= \bigcap_{T=t_1}^{\infty}\overline{\varDelta E^{t_1,\ldots,t_{N(T)-1},s(T)}}.$$
This is an intersection of decreasing closed sets, and so a compact set.
For each $\omega\in  E^{t_1,\ldots,t_{N-1},s}$ fix
a point $x(\omega)\in \varDelta_\omega\cap \Gamma$. 
Let $\nu_T$ denote the uniform distribution on the finite set $\bigcup\{x(\omega)\colon\omega\in  
E^{t_1,\ldots,t_{N-1},s} \}.$
Let $\nu$ be an accumulation point of the sequence $\{\nu_T\}_{T=t_1}^\infty$
in the weak*-topology on the space of measures on $M$.
We show that
\begin{equation}\label{mass}
\liminf_{r\to 0}\frac{\log \nu (B(x,r)\cap\varDelta)}{\log r} 
\geq \lim_{i\to\infty}\dim(\mu_i)\quad\forall x\in \Gamma,
\end{equation}
where $B(x,r)$ denotes the connected open subset of $M$ of length $2r$ centered at $x$.
From the mass distribution principle and the compactness of $\Gamma$ it follows that
\begin{equation}\label{eqs}{\rm dim}_H\Gamma\geq\lim_{i\to\infty}\dim(\mu_i).\end{equation}

In order to show \eqref{mass}, for each integer $T\geq t_1$ in \eqref{k}
 define \begin{equation*}
r_{T}= \exp\left(-t_{N-1}\ell_{N-1}\left(\chi(\mu_{N-1})+ \frac{3}{N-1}\right)
-s\ell_{N}\left(\chi(\mu_N)+ \frac{2}{N}\right)\right).\end{equation*}
\begin{lemma}\label{choose}
One can choose $\{t_i\}_{i=1}^\infty$ inductively
so that $r_T\searrow 0$ as $T\nearrow\infty$.
\end{lemma}
\begin{proof}If $s<t_N-1$
then $N(T+1)=N(T)$ and $s(T+1)=s(T)+1$, and so
$$\frac{r_{T+1}}{r_T}=\exp\left(-\ell_{N}\left(\chi(\mu_N)+ \frac{2}{N}\right)\right)<1.$$
If $s=t_N-1$, then $N(T+1)=N(T)+1$ and $s(T+1)=0$, and so
$$
r_T= \exp\left( -t_{N-1}\ell_{N-1}\left(\chi(\mu_{N-1})+
\frac{3}{N-1}\right)
-(t_{N}-1)\ell_{N}\left(\chi(\mu_{N})+ \frac{2}{N}\right)\right),$$
and
$$r_{T+1}= \exp\left(-t_{N}\ell_{N}\left(\chi(\mu_{N})+ \frac{3}{N}\right)\right).$$
Therefore, by choosing $\{t_i\}_{i=1}^\infty $ inductively such that  $t_i$ is large enough
compared to $t_{i-1}$ one can make sure that
$r_{T+1}/r_T\leq1/2$ for every $T\geq t_1$.
\end{proof}

Let $r\in(0,r_{t_1}]$. Then there exists $T\geq t_1$ such that
\begin{equation}\label{ak}r_{T+1}<r\leq r_T.\end{equation}
For each $\omega\in E^{t_1,\ldots,t_{N-1},s}$ we have
\begin{equation}
\begin{split}\label{suplog}
\sup_{\varDelta_\omega}\log|(f^{|\omega|})'|&\leq\sum_{i=1}^{N-1} t_{i}\ell_{i}\left(\chi(\mu_{i})+\frac{2}{i}\right)+s \ell_{N}
\left(\chi(\mu_{N})+\frac{2}{N}\right) \\
&\leq t_{N-1}\ell_{N-1}\left(\chi(\mu_{N-1})+\frac{3}{N-1}\right)+s \ell_{N}
\left(\chi(\mu_{N})+\frac{2}{N}\right).
\end{split}
\end{equation}
The first inequality follows from \eqref{eq-ell} and \eqref{qe2}.
The second one holds provided
 each $t_i$ is chosen large enough
compared to $t_1,\ldots,t_{i-1}$.

     \begin{lemma}\label{inf-image}
  $\displaystyle{\inf_{\omega\in E^*} |f^{|\omega|}\varDelta_\omega|}>0.$
 \end{lemma}
 \begin{proof}
 The infimum is equal to $\inf_{a\in S}|f\varDelta_a|$.
 Let $b\in E^*$. For each $a\in S$ there exists $\lambda\in\Lambda$
 such that $a\lambda b\in E^*$. Take one connected component of $(f^{|\lambda|+1}|_{\varDelta_a})^{-1}\varDelta_b$
 and denote it by $A$. Since $f^{|\lambda|}(fA)=\varDelta_b$, the mean value theorem implies
 $\sup_{[\lambda]}|(f^{|\lambda|})'||fA|\geq |\varDelta_b|$. Since $\Lambda$ is a finite set,
 $|f\varDelta_a|\geq|fA|\geq|\varDelta_b|/\sup_{\lambda\in\Lambda}\sup_{[\lambda]}|(f^{|\lambda|})'|>0$
 as required.
 \end{proof}

Let $\kappa>0$ denote the infimum in Lemma \ref{inf-image}. 
By the mean value theorem, Lemma \ref{inf-image} and \eqref{suplog},
\begin{align*}
|\varDelta_{\omega}|&\geq \frac{|f^{|\omega|}\varDelta_\omega|}{\sup_{\varDelta_\omega}|(f^{|\omega|})'|}\\
&\geq\kappa\exp\left(-t_{N-1}\ell_{N-1}\left(\chi(\mu_{N-1})+\frac{3}{N-1}\right)-s\ell_N\left(\chi(\mu_{N})+\frac{2}{N}\right)\right)\\
&=\kappa r_T.
\end{align*}
Hence, for any $x\in \Gamma$  we have 
 \begin{equation}
 \begin{split}\label{covering}
   \#\left\{\omega \in E^{t_1,\dots ,t_{N-1},s}   \colon B(x,r)\cap\Delta_\omega\neq \emptyset
   \right\}&\le \frac{2r_{T}}{\inf
_{\omega\in   E^{t_1,\ldots,t_{N-1},s}} |\varDelta_{\omega}|}+2\\
&\leq 2(\kappa ^{-1}+1).
 \end{split}
 \end{equation}
\begin{lemma}\label{boundary}
For every $\omega\in E^{t_1,\ldots,t_{N-1},s}$ we have
$\nu(\partial\varDelta_\omega)=0$.
\end{lemma}
\begin{proof}
By construction, each set at the $T$-th level contains a finite number of sets at the $(T+1)$-th level, and this number of branches is independent of the sets at the
 $T$-th level.
From the positivity of entropy \eqref{equa1}
and \eqref{qe1} with $i=N-1$ we have 
 $\#F^{\ell_{N-1}}>1$. Hence,
the number of branches at the $T$-th lvel in the case $s<t_N-1$ 
is $\#F^{\ell_{N-1},t_{N-1}}>1$.
A similar argument shows that the number of branches
in the case $s=t_N-1$ is
$\#F^{\ell_{N},t_{N}}>1$.
For every $\omega\in E^{t_1,\ldots,t_{N-1},s}$
and for every $p\geq T$,
$$\nu_p(\varDelta_\omega)
=\frac{1}{\#(F^{\ell_1})^{t_1}\cdots
\#(F^{\ell_{N-1}})^{t_{N-1}}
\#(F^{\ell_N})^{s}}.$$
Since the number of branches is strictly bigger than $1$ and cylinders of fixed length have pairwise disjoint interiors, the claim follows.
\end{proof}
From Lemma \ref{boundary}, the weak*-convergence of measures gives
$\nu(\varDelta_{\omega})=\lim_{p\to\infty}\nu_p(\varDelta_{\omega})$
for every $\omega\in E^{t_1,\ldots,t_{N-1},s}$.
By  \eqref{covering} we then have
\begin{align*}
\nu(B(x,r)\cap\Gamma)\leq&\frac{2(\kappa ^{-1}+1)}{   (\#F^{\ell_{N-1}})^{t_{N-1}}(\# F^{\ell_N})^{s}    }\\
\leq& 2(\kappa ^{-1}+1)\times\\
&\exp\left(-t_{N-1}\ell_{N-1}\left(h(\mu_{N-1})-\frac{1}{N-1}\right)-
s\ell_{N}\left(h(\mu_{N})-\frac{1}{N}\right)\right).\end{align*}
By the definition of $r_T$ and  \eqref{ak} this yields
\begin{align*}
\frac{\log\nu (B(x,r)\cap\Gamma)}{\log r}\geq&\frac{\log\nu (B(x,r)\cap\Gamma)}{\log r_T}\\
\geq&\frac{t_{N-1}\ell_{N-1}\left(h(\mu_{N-1})-
1/(N-1)\right)+s\ell_{N}\left(h(\mu_{N})-1/N\right)}
{t_{N-1}\ell_{N-1}\left(\chi(\mu_{N-1})+3/(N-1)\right)
+s\ell_{N}\left(\chi(\mu_{N})+2/N\right)}\\
&+\frac{\log(2(\kappa ^{-1}+1))}{\log r_T}.\end{align*}
We have $T\to\infty$ as  $r\to 0$, and so  $N \to\infty$ by  \eqref{k}.
 Hence \eqref{mass} follows.
\medskip

 \noindent{\it Step 2: Verification of $\Gamma\subset B(\boldsymbol\phi,\boldsymbol\alpha)$.}
 It remains to fix $\{t_i\}_{i=1}^\infty$ so that
  $\Gamma\subset B(\boldsymbol\phi,\boldsymbol\alpha)$ holds,
  namely for all $x\in\Gamma$,
  $$\lim_{n\to\infty}\left\|\frac{1}{n}S_n\boldsymbol\phi_k(x)-\boldsymbol\alpha_k\right\|=0\quad\forall k\geq1.$$
For each $i\ge 2$ we choose $t_i$
large enough compared to $t_1,\ldots,t_{i-1}$ so that the following holds in addition to all the previously used inequalities:

\begin{equation}\label{loweq2}
\sum_{j=1}^i\frac{3t_j\ell_j}{j}\leq\frac{4t_i\ell_i}{i};
\end{equation}

\begin{equation}\label{loweq0}c_{i-1}+
\sup_{1\leq p\leq \ell_i}\sup_{\varDelta F^{\ell _i}}\|S_p\boldsymbol\phi_i\|
+(|\Lambda|+\ell_i)\left\|\boldsymbol\alpha_i\right\|\leq t_{i-1}.\end{equation}

 Let $k\geq1$.
  Let $T\geq t_1$ be such that $k\leq N-1$. 
  \begin{lemma}\label{equation-a}
Let $i\geq1$, $t\geq1$ be integers such that $\eta_i\in F^{\ell_i,t}$. Then
\begin{equation*}
\sup_{\varDelta_{\eta_i\lambda_{b_ia_{i+1}}}}\|S_{|\eta_i\lambda_{b_ia_{i+1}}|}\boldsymbol\phi_k-|\eta_i\lambda_{b_ia_{i+1}}|\boldsymbol\alpha_k\|\leq\frac{2t\ell_i}{i}+tc_k+t|\Lambda|\cdot\|\boldsymbol\alpha_k\|.\end{equation*}
 \end{lemma}
 \begin{proof}
 The definition of $c_k$ gives
 \begin{equation}\label{qe-3} \sup_{\varDelta_{\lambda_{b_ia_{i+1}}}}\|
S_{|\lambda_{b_ia_{i+1}}|}\boldsymbol\phi_k-|\lambda_{b_ia_{i+1}}|\boldsymbol\alpha_k \|<c_k+|\Lambda|\cdot\|\boldsymbol\alpha_k\|.\end{equation}
 The desired inequality for $t=1$ is a consequence of \eqref{equa2}, \eqref{qe3} and \eqref{qe-3}.
 In the case $t\geq2$, we have
 $\eta_i=\omega_1 \lambda_{b_ia_i} \omega_2\cdots  \lambda_{b_ia_i}\omega_{t}$ with
$\omega_1,\ldots,\omega_{t}\in F^{\ell_i}$ and $\lambda_{b_ia_i}\in\Lambda$.
 By \eqref{equa2} and \eqref{qe3}, for every $1\leq j\leq t-1$ we have
\begin{align*}\sup_{\varDelta_{\omega_j\lambda_{b_ia_i}}}\|
S_{|\omega_j\lambda_{b_ia_i}|}\boldsymbol\phi_k-|\omega_j\lambda_{b_ia_i}|\boldsymbol\alpha_k \|\leq&
\sup_{\varDelta_{\omega_j}}\|
S_{\ell_i}\boldsymbol\phi_k-\ell_i\boldsymbol\alpha_k \|
+\sup_{\varDelta_{\lambda_{b_ia_i}}}\|
S_{|\lambda_{b_ia_i}|}\boldsymbol\phi_k-|\lambda_{b_ia_i}|\boldsymbol\alpha_k \|\\
\leq&\sup_{\varDelta_{\omega_j}}\left\|
S_{\ell_i}\boldsymbol\phi_k -\ell_i\int\boldsymbol\phi_k d\mu_i\right\|+\left\|\ell_i\int\boldsymbol\phi_k d\mu_i-\ell_i\boldsymbol\alpha_k\right\|\\
&+c_k+|\Lambda|\cdot\|\boldsymbol\alpha_k\|\\
\leq&\frac{2\ell_i}{i}+c_k+|\Lambda|\cdot\|\boldsymbol\alpha_k\|.\end{align*}
Summing this inequality over all $1\leq j\leq t$
yields the desired inequality.
The case $t=1$ is covered by the same argument.
 \end{proof}

  Recall that each
 $\omega\in E^{t_1,\ldots,t_{N-1},s}$ has the form 
 $$\omega= \eta_1\lambda_{b_1a_2}   \cdots \eta_{k-1}\lambda_{b_{k-1}a_k}\cdots \eta_{N-1}\lambda_{b_{N-1}a_{N}}\eta_N,$$ with
 $\eta_i\in F^{\ell_i,t_i}$, $\lambda_{b_ia_{i+1}}\in\Lambda$ for $1\leq i\leq N-1$ and
 $\eta_N\in F^{\ell_N,s}$. Use the triangle inequality and then split
$$\sup_{\varDelta_\omega}\|S_{|\omega|}\boldsymbol\phi_k-|\omega|\boldsymbol\alpha_k\|\leq I+I\!I,$$ where 
\begin{align*}
I=\sup_{\varDelta_{\omega_I}}\|S_{| \omega_I|}\boldsymbol\phi_k-| \omega_I|\boldsymbol\alpha_k\|\ \text{ and }\
I\!I=\sup_{\varDelta_{\omega_{I\!I}}}\|S_{|\omega_{I\!I}|}\boldsymbol\phi_k-|\omega_{I\!I}|\boldsymbol\alpha_k\|,
\end{align*}
and $\omega_I$, $\omega_{I\!I}\in E^*$ are given by
$$\omega_I= \eta_1\lambda_{b_1a_2}   \cdots \eta_{k-1}\lambda_{b_{k-1}a_k}\ \text{ and }\ \omega_{I\!I}= \eta_k\lambda_{b_ka_{k+1}}   \cdots \eta_{N-1}\lambda_{b_{N-1}a_N}\eta_N.$$
Applying Lemma \ref{equation-a} to $\eta_i\lambda_{b_ia_{i+1}}$ for $k\leq i\leq N-1$ in $\omega_{I\!I}$
 and
using once more Lemma \ref{equation-a} to deal with the last word $\eta_N$ yields
\begin{equation} \label{equation-b}
\begin{split}
I\!I\leq&
\sum_{i=k}^{N-1}\left(\frac{2t_i\ell_i}{i}+t_ic_i+t_i|\Lambda|\cdot\|\boldsymbol\alpha_k\| \right)+
s\left(\frac{2\ell_N}{N}+c_N+|\Lambda|\cdot\|\boldsymbol\alpha_k\|\right)\\
=&\sum_{i=k}^{N-1}\frac{2t_i\ell_i}{i}+\sum_{i=k}^{N-1}t_i\left(c_i+|\Lambda|\cdot\|\boldsymbol\alpha_k\|\right)+s\left(\frac{2\ell_N}{N}+c_N+|\Lambda|\cdot\|\boldsymbol\alpha_k\|\right)\\
\leq&\sum_{i=k}^{N-1}\frac{2t_i\ell_i}{i}+t_{N-1}(N-1)\left(c_{N-1}+
|\Lambda|\cdot\|\boldsymbol\alpha_k\|\right)+s\left(\frac{2\ell_N}{N}+c_N+|\Lambda|\cdot\|\boldsymbol\alpha_k\|\right)\\
\leq&\sum_{i=k}^{N-1}\frac{2t_i\ell_i}{i}+\frac{t_{N-1}\ell_{N-1}}{N-1}+\frac{3s\ell_N}{N}
\ \text{ by \eqref{important2}}\\
\leq&\sum_{i=k}^{N-1} \frac{3t_{i}\ell_{i}}{i}+\frac{3s\ell_N}{N}.
\end{split}\end{equation}
The same argument gives
$$I\leq\sum_{i=1}^{k-1}\frac{3t_i\ell_i}{i}.$$
  Now, let $x\in \Gamma$.
  For each integer $n\geq t_1\ell_1$ there exists an integer $T\geq t_1$ such that
   if
  $\omega\in E^{t_1,\ldots,t_{N-1},s}$
 and $x\in\varDelta_\omega$, then the following holds:
 \begin{equation}\label{equation-c}0\leq n-|\omega|\leq|\Lambda|+\ell_N;\end{equation}
\begin{equation}\label{equation-d} 
\|S_{n-|\omega|}\boldsymbol\phi_k(\sigma^{|\omega|}x)\|\leq c_{N-1}+\sup_{1\leq p\leq \ell_N}\sup_{\varDelta F^{\ell _N}}\|S_p\boldsymbol\phi_k\|.
\end{equation}
From \eqref{equation-b}, \eqref{equation-c}, \eqref{equation-d} we have
\begin{align*}
\|S_n\boldsymbol\phi_k(x)-n\boldsymbol\alpha_k\|\leq&\|S_{|\omega|}\boldsymbol\phi_k(x)-|\omega|\boldsymbol\alpha_k\|
+\|S_{n-|\omega|}\boldsymbol\phi_k(\sigma^{|\omega|}x)-(n-|\omega|)\boldsymbol\alpha_k\|\\
\leq&\sum_{i=1}^{N-1} \frac{3t_{i}\ell_{i}}{i}+\frac{3s\ell_N}{N}
+ c_{N-1}\\
&+\sup_{1\leq p\leq \ell_N}\sup_{\varDelta F^{\ell _N}}\|S_p\boldsymbol\phi_k\|+
(|\Lambda|+\ell_N)\|\boldsymbol\alpha_k\|.\end{align*}
On the right-hand side, we use \eqref{loweq2} to bound the first two terms, and \eqref{loweq0} to bound the remaining terms. We have 
\begin{align*}
\|S_n\boldsymbol\phi_k(x)-n\boldsymbol\alpha_k\|
\leq&\frac{4t_{N-1}\ell_{N-1}}{N-1}+\frac{3s\ell_N}{N}+t_{N-1}\\
\leq&\frac{4n}{N-1}+\frac{t_{N-1}\ell_{N-1}}{N-1}\ \text{  by \eqref{loweq1}, \eqref{equation-c}}\\
\leq&\frac{5n}{N-1}\ \text{ by \eqref{equation-c}}.
\end{align*}
Since $N$ increases as $n$ does, it follows that $\lim_{n\to\infty}\|(1/n)S_n\boldsymbol\phi_k(x)-\boldsymbol\alpha_k\|=0$.
Since $k\geq1$ is arbitrary, this completes the proof of Proposition \ref{lowd}.
\end{proof}

 \subsection{Proof of the Main Theorem}\label{end}
 Let $f\colon\varDelta\to M$ be a non-uniformly expanding Markov map which is finitely irreducible, has mild distortion and uniform decay of cylinders. 
  Let $\boldsymbol\phi\in\mathcal{F}^{\mathbb N}$ and
   $\boldsymbol\alpha\in\mathbb R^{\mathbb N}$.

\subsubsection*{Proof of Main Theorem(a)}  Assume that
     for any $\epsilon>0$ and any integer $k\geq1$ there exists a measure  
     $\mu\in\mathcal M(f)$ such that
   $\left\|\int\boldsymbol\phi_k d\mu-\boldsymbol\alpha_k\right\|<\epsilon.$
   Then there exists an expanding measure $\mu' \in \mathcal{M}(f)$
   such that
   $\left\|\int\boldsymbol\phi_k d\mu'-\boldsymbol\alpha_k\right\|<\epsilon.$
  By Lemma \ref{nonergodic}, there exists an ergodic expanding measure $\mu''$
   such that
   $\left\|\int\boldsymbol\phi_k d\mu''-\boldsymbol\alpha_k\right\|<\epsilon.$
   Proposition \ref{lowd} shows
     $B(\boldsymbol\phi,\boldsymbol\alpha)\neq\emptyset.$
     Conversely, if $x\in B(\boldsymbol\phi,\boldsymbol\alpha)$
     then using the finite irreducibility and $D_1(\phi_i)<\infty$ for each $i\geq1$,
     one can find for any $\epsilon>0$ and any integer $k\geq1$ a measure 
     $\mu\in\mathcal M(f)$ supported on a single periodic orbit 
     such that
     $\left\|\int\boldsymbol\phi_k d\mu-\boldsymbol\alpha_k\right\|<\epsilon.$ 
      \qed

\subsubsection*{Proof of Main Theorem(b)}
Assume \eqref{dim} and $B(\boldsymbol\phi,\boldsymbol\alpha)\neq\emptyset$.
Put $$d_\infty=\displaystyle{\lim_{k\to\infty}d_k},$$
where $d_k$ is the constant defined in Proposition \ref{up}.
Below we will verify $b_{\boldsymbol\phi}(\boldsymbol\alpha)=d_\infty$
by showing that 
  the upper bound obtained in Section \ref{bsec} and
  the lower bound in Section \ref{keylow} coincide.
  We then show that if all $\phi_i$ are bounded then the constant $\beta_\infty$
  is a lower bound for $b_{\boldsymbol\phi}(\boldsymbol\alpha)$.  Lastly we show the  remaining assertions.

 We begin by claiming that
  \begin{equation}\label{lowbd}
  b_{\boldsymbol\phi}(\boldsymbol\alpha)\geq d_\infty.\end{equation}
 Indeed, by the definition of $d_\infty$, for every $c>0$ and every $k\ge 1$ there exists  $\mu_k\in\mathcal M(f)$ such that
$\dim(\mu_k)>d_\infty -c$ and
$\left\|\int\boldsymbol\phi_k d\mu_k-\boldsymbol\alpha_k\right\|<1/k.$ We may assume that each $\mu_k$ is expanding and, by Lemma \ref{nonergodic}, also  ergodic.  Then Proposition \ref{lowd} yields $b_{\boldsymbol\phi}(\boldsymbol\alpha)\geq d_\infty - c$. Since $c>0$ was arbitrary, we have thus  shown
$b_{\boldsymbol\phi}(\boldsymbol\alpha)\geq d_\infty$.


To prove  the reverse inequality we distinguish  two cases. First, 
we assume $\sup_\varDelta\|\boldsymbol\phi_{k_0}\|=\infty$ for some $k_0\geq1$.
The saturation \eqref{dim} implies
 $b_{\boldsymbol\phi}(\boldsymbol\alpha)\leq\delta_0$.
If this inequality is strict, then we have 
 $d_\infty\le b_{\boldsymbol\phi}(\boldsymbol\alpha)<\delta_0$ 
from \eqref{lowbd}, and
Proposition \ref{up}(a) gives $b_{\boldsymbol\phi}(\boldsymbol\alpha)\leq 
d_\infty$.  We claim that if $b_{\boldsymbol\phi}(\boldsymbol\alpha)=\delta_0$
then
$d_\infty=\delta_0$ holds,
for otherwise $ d_\infty<\delta_0$ and so there exists $\epsilon>0$ such that
$d_\infty+\epsilon<\delta_0$. Hence, we have 
$d_k<d_\infty +\epsilon<\delta_0$ for $k\geq k_0$ large enough and
Proposition \ref{up}(a)  gives the desired contradiction 
 $b_{\boldsymbol\phi}(\boldsymbol\alpha)
 \leq d_\infty+\epsilon<\delta_0$.
 
 Next we assume $\sup_\varDelta\|\boldsymbol\phi_k\|<\infty$ for every $k\geq1$.
 A slight modification of the argument in the previous case 
 with Proposition \ref{up}(b) shows
  $b_{\boldsymbol\phi}(\boldsymbol\alpha)
 \leq \max\{d_\infty,\beta_\infty\}.$
The reverse inequality 
 and the equalities $b_{\boldsymbol\phi}(\boldsymbol\alpha)
 = \max\{d_\infty,\beta_\infty\}=d_\infty$
are consequences of \eqref{lowbd} and the next lemma which is an adaptation of 
\cite[Lemma 5.1]{FanJorLiaRam16} to our setting.
\begin{lemma}\label{fan'}
Let $k\geq1$ be an integer, $\epsilon>0$ and $\mu\in\mathcal M(f)$ be an expanding measure
be such that
$$\left\|\int\boldsymbol\phi_kd\mu-\boldsymbol\alpha_k\right\|<\epsilon.$$
If  $\sup_\varDelta\|\boldsymbol\phi_k\|<\infty$, then there exists an expanding measure $\nu\in\mathcal M(f)$ such that
$$\dim(\nu)>\beta_\infty-\epsilon\ \text{ and }\
\left\|\int\boldsymbol\phi_kd\nu-\boldsymbol\alpha_k\right\|<2\epsilon.$$
\end{lemma}
\begin{proof}
One can adapt \cite[Lemma 2.5]{FanJorLiaRam16} to our setting
to show that there exists a sequence $\{\mu_n\}_{n=1}^\infty$
of expanding measures which are supported on a finite union of cylinders and such that $\lim_{n\to\infty}\dim(\mu_n)=\beta_\infty$
and $\lim_{n\to\infty}\chi(\mu_n)=\infty$.
Using this and arguing along the line of the proof of
 \cite[Lemma 5.1]{FanJorLiaRam16} proves the assertion of the lemma.
\end{proof}
We are almost ready to complete the proof of the Main Theorem(b).
For the proof of the remaining assertions note that,
for each fixed $k\geq1$ we have $$\dim_HB(\boldsymbol\phi,\boldsymbol\alpha_k)=\lim_{\epsilon\to0}\sup\left\{
\dim(\mu)\colon\mu\in\mathcal M(f),\ \left\|\int\boldsymbol\phi_k d\mu-\boldsymbol\alpha_k \right\|<\epsilon\right\}.$$
Therefore $b_{\boldsymbol\phi}(\boldsymbol\alpha)=\dim_HB(\boldsymbol\phi,\boldsymbol\alpha)=\lim_{k\to\infty}\dim_HB(\boldsymbol\phi,\boldsymbol\alpha_k)$ holds,
proving the second assertion in (b).

Now assume in addition that $f$ has a neutral periodic point. 
Since $f$ is saturated,
there is an expanding measure with dimension arbitrarily close to $\delta_0$. Considering convex combinations
of such measures and measures supported on single neutral periodic orbits, we obtain an expanding measure with dimension arbitrarily close to $\delta_0$
and Lyapunov exponent arbitrarily close to $0$. From this and 
the first assertion in (b), the last one in (b) follows.
\qed

\section{Verifications of the assumptions in the Main Theorem}
In this section we provide sufficient conditions for the uniform decay of cylinders
and the saturation by expanding measures assumed in the Main Theorem.
\subsection{Uniform decay of cylinders}\label{vdecay}
Let $f\colon\varDelta\to M$ be a non-uniformly expanding Markov map.
Inspired by \cite[Section 8]{MauUrb03},
we introduce the following condition:
 \begin{itemize}

      \item[(M3)] 
  there exists $s>1$ such that if  $a_1,a_2\in S$
  are such that $f \varDelta_{a_1}\supset \varDelta_{a_2}$ and
  $a_2$ is an expanding index,  or else $a_{1}\neq a_2$,
 then $\inf_{\varDelta_{a_1a_2}}|(f^2)'|\geq s.$
 
 \end{itemize}



\begin{lemma}\label{uniform}
Let $f\colon\varDelta\to M$ be a non-uniformly expanding Markov map.
 If (M3) holds, then $f$ has uniform decay of cylinders.
\end{lemma}
For a proof of Lemma \ref{uniform} we need the next lemma. For $a\in S$ such that $aa$ is admissible and $n\geq1$, let $a^n\in E^n$ denote the $n$-string of $a$.
\begin{lemma}\label{shrink}
Let $a\in \Omega$ be such that $aa$ is admissible. Then $|\varDelta_{a^n}|\to0$ as $n\to\infty$.  
 \end{lemma}
 \begin{proof}
  Put $I=\bigcap_{n=1}^\infty\varDelta_{a^n}$ and assume $|I|>0$ by contradiction. Since $f(I)=I$ and $|f'|\geq1$ everywhere on $I$, $|f'|\equiv1$ on $I$.
  Since there are only finitely many neutral points, we obtain a contradiction.
\end{proof}

\begin{proof}[Proof of Lemma \ref{uniform}]
We slightly modify the proof of \cite[Lemma 8.1.2]{MauUrb03}.
Let $$g(n)=\max\{|\varDelta_{a^n}|\colon a\in \Omega, \text{ $aa$ is admissible}\}.$$
The finiteness of $\Omega$ and
Lemma \ref{shrink} give $\lim_{n\to\infty}g(n)=0$.
Let $\omega\in E^n$. 
Look at the longest block of the same index in $\Omega$ appearing in $\omega$.
If the length of this block exceeds $\sqrt{n}$ then we have 
$|\varDelta_{\omega}|\leq g(\lfloor\sqrt{n}\rfloor)$ by (M2).
Otherwise, we have 
$|\varDelta_{\omega}|\leq s^{(\lfloor\sqrt{n}\rfloor -1)/2}$ by (M3). Hence, $f$ has uniform decay of cylinders.
\end{proof}

\begin{remark}
 Condition (M3) is not necessary for the uniform decay of cylinders.
The Farey map $f\colon [0,1]\to[0,1]$ given by
$fx=\frac{x}{1-x}$ for $0\leq x< \frac{1}{2}$ and
$fx=\frac{1-x}{x}$ elsewhere
has a Markov partition $\{[0,\frac{1}{2}),[\frac{1}{2},1]\}$. 
It does not satisfy (M3) and has uniform decay of cylinders
by Lemma \ref{shrink}.
\end{remark}
\subsection{Saturation by expanding measures}\label{dist-cont}
One sufficient condition for
\eqref{dim} is the existence of
an induced Markov map which captures sets of full Hausdorff dimension, 
some power of which is uniformly expanding, finitely irreducible and
satisfies R\'enyi's condition. Then, the thermodynamic formalism for the induced map
and a finite approximation technique
provide us with expanding measures of dimension arbitrarily close to $\dim_HJ$. However, 
to verify R\'enyi's condition
for induced maps usually requires explicit formulas of individual maps (see e.g., \cite[Lemma 2.7 and Lemma 2.8]{BowSer79}, \cite[Lemma 3.4]{KesStr04}). 
Below we provide a convenient sufficient condition for the saturation \eqref{dim}, 
replacing R\'enyi's condition by conditions in terms of
the Markov partition of the original map.

\begin{prop}\label{verify}
Let $f\colon\varDelta\to M$ be a non-uniformly expanding Markov map with $C^2$ branches
 satisfying R\'enyi's condition and (M3). Assume that
\begin{itemize}
 \item[(i)] if $a\in \Omega$, $aa$ is admissible and 
  $\varDelta_a$ contains no neutral fixed point, then $\varDelta_a$ is closed,
  \end{itemize}
  and moreover, there exist a subset $F$ of $E^*$ and a  function
$\tau\colon \{\varDelta_\omega\}_{\omega\in F}\to \mathbb N\setminus\{0\}$ 
such that the following holds:
  \begin{itemize}
\item[(ii)] the cylinders $\{\varDelta_\omega\}_{\omega\in F}$ 
have disjoint interiors,
and for each $\omega\in F$ and $0\leq n\leq\tau(\varDelta_\omega)-1$, 
$f^n\varDelta_\omega$
does not contain a neutral
fixed point;

\item[(iii)] 
the induced map $$\tilde f\colon x\in K\mapsto f^{\tau(\varDelta(x))}x\in M$$ where
 $K=\bigcup_{\omega\in F}\varDelta_\omega$ and
$\varDelta(x)$ is an element containing $x$,
is a finitely irreducible non-uniformly expanding Markov map
with a Markov partition $\{\varDelta_\omega\}_{\omega\in F}$;
\item[(iv)] we have $$\dim_HJ=\dim_H\left(\bigcap_{n=0}^\infty \tilde f^{-n}K\right).$$
     \end{itemize}
Then the following holds:
\begin{itemize}
\item[(a)] there exists $C>0$ such that for every $\omega\in F$  and all $x,y\in\varDelta_\omega$,
\begin{equation*}\label{controldist}\log\frac{(\tilde f)'x}
{(\tilde f)'y}\leq C|\tilde fx-\tilde fy|;\end{equation*}

    \item[(b)] 
there exists an ergodic expanding measure in $\mathcal M(f)$ with dimension arbitrarily close to $\dim_HJ$.
In particular, $f$ is saturated;

\item[(c)] $\tilde f$ has mild distortion and is saturated, and $(\tilde f)^2$ is uniformly expanding.

\end{itemize}
\end{prop}

\noindent{\it Proof.}
In order to treat neutral indices we need a couple of lemmas.
The next one controls distortion of a branch whose domain contains a neutral fixed point.

\begin{lemma}\label{scope'}
Let $g\colon[0,1]\to \mathbb R$ be a $C^{2}$ map satisfying
$g(0)=0$, $g'(0)=1$ and $g'>1$ on $(0,1]$.
There exists a constant $C>0$ such that for any $n\geq1$
and all $x$, $y\in (g^{-n}(1),g^{-n+1}(1)]$ we have
$$\log\frac{(g^{n})'(x)}{(g^n)'(y)}\leq C|g^n(x)-g^n(y)|.$$
\end{lemma}
\begin{proof}
Put $C_0=\sup_{[0,1]}|g''|$
and $I_n=(g^{-n}(1),g^{-n+1}(1)]$. For $x,y\in I_n$ and
$0\leq k\leq n-1$ we have 
\begin{align*}
\log\frac{g'(g^k(x))}{g'(g^k(y))}
\leq|g'(g^k(x))-g'(g^k(y))|\leq C_0|I_{n-k}|.\end{align*}
Therefore
   \begin{equation*}\label{lem}\sup_{x,y\in I_n}\log\frac{(g^{n-k})'(g^k(x))}{(g^{n-k})'(g^k(y))}\leq C_0,\end{equation*}
and
\begin{equation*}
\begin{split}
\log\frac{g'(g^k(x))}{g'(g^k(y))}
&\leq C_0\frac{|g^k(x)-g^k(y)|}{|I_{n-k}|}|I_{n-k}|\leq \frac{C_0e^{2C_0}}{|g(I_{1})|}|g^n(x)-g^n(y)||I_{n-k}|.\end{split}
\end{equation*}
Put $C=C_0e^{2C_0}/|gI_1|$. 
Summing this over all $0\leq k\leq n-1$ yields
the desired inequality. 
\end{proof}

The next lemma will be used to treat the occurrence of the same neutral index consecutively
 which does not correspond to a neutral fixed point.
\begin{lemma}\label{g}
There exists $\rho>1$  such that
if $a\in \Omega$ is such that $aa$ is admissible and
 $\varDelta_{a}$ does not contain a neutral fixed point, then
$$\inf_{\varDelta_{aa}}|(f^2)'|\geq\rho.$$
Moreover,  $(\tilde f)^2$ is uniformly expanding. 
\end{lemma}
\begin{proof}
Since
 $\varDelta_{a}$ is closed by the assumption (i) in Proposition \ref{verify}
and $|f'|\geq1$,
there is a fixed point in $\varDelta_a$.
By the assumption of Lemma \ref{g}  this fixed point is not neutral. Further, since $|f'|\geq1$, there is no other 
fixed point in
$\varDelta_{a}$. It follows that  $|(f^{2})'|$ is bounded away from $1$ on $\varDelta_{aa}$ for otherwise, since $\varDelta_a$ is closed, there would exist a neutral periodic point in $\varDelta_a$ of period two.
The uniform expansion of $(\tilde f)^2$ follows from (M3).
\end{proof}

To prove Proposition \ref{verify}(a), let $\omega\in E^*$.
We decompose $\omega$
into a concatenation of admissible words
$\omega=\omega(0)\omega(1)\cdots\omega(k)$ 
with the following properties:
\begin{itemize}
    \item
    each 
$\omega(i)$ $(0\leq i\leq k)$ is one of the following types: 
(I) a string of expanding indices;
 (II)
a single neutral index $a$
for which $\varDelta_a$ contains a neutral fixed point;
(III) a single neutral index $a$
for which $\varDelta_a$ does not contain a neutral fixed point;
\item 
if $k'\leq k$ and
$\omega=\omega'(1)\cdots
\omega'(k')$
is another decomposition with the first property, 
then $k=k'$ and $\omega(i)=\omega'(i)$ for every $0\leq i\leq k$.
\end{itemize}
We claim that there exists a constant $C>0$ which is independent of $\omega$ such that the following holds for every $0\leq i\leq k$ and all $x,y\in\varDelta_{\omega(i)}$:
\begin{equation}\label{control1}
    \log\frac{(f^{|\omega(i)|})'x}{(f^{|\omega(i)|})'y}\leq C|f^{|\omega(i)|}x-f^{|\omega(i)|}y|.
\end{equation}
Indeed, an inequality of the form \eqref{control1} in the case (I) applies to $\omega(i)$ follows from R\'enyi's condition and (M3).
For those $i$ such that (II) applies to $\omega(i)$,
we use Lemma \ref{scope'}. For those $i$
such that (III) applies to $\omega(i)$, we use
R\'enyi's condition, Lemma \ref{g}
and the finiteness of the number of neutral indices.

Assume $k\geq1$. 
Condition (M3) implies
\begin{equation}\label{expansion1}\inf_{\varDelta_{\omega(i)
\omega(i+1)}}|(f^{|\omega(i)
\omega(i+1)|})'|\geq s\quad\text{
$0\leq \forall i\leq k-1$}.\end{equation}
Now, let
$x,y\in \varDelta_\omega$. 
The chain rule gives
\begin{equation*}\log\frac{(\tilde f)'x}
{(\tilde f)'y}=
\log\frac{(f^{|\omega(0)|})'x}{(f^{|\omega(0)|})'y}+\sum_{i=1}^{k}\log\frac{(f^{|\omega(i)|})'f^{|\omega(0)\cdots\omega(i-1)|}x}
{(f^{|\omega(i)|})'f^
{|\omega(0)\cdots\omega(i-1)|}y}.\end{equation*}
Applying \eqref{control1} to the right-hand side,
we have
\begin{equation}\label{cont1}
\log\frac{(\tilde f)'x}
{(\tilde f)'y}\leq C\sum_{i=1}^{k+1}|f^{|\omega(0)\cdots\omega(i-1)|}x-f^
{|\omega(0)\cdots\omega(i-1)|}y|.\end{equation}
We apply the mean value theorem to each term in the series and use the uniform expansion
 $\inf_{\varDelta_{\omega(i)\cdots\omega(k)}}|(f^{|\omega(i)\cdots\omega(k)|})'|\geq s^{k-i+1}$
from  \eqref{expansion1} to deduce that
\begin{align*}
\log\frac{(\tilde f)'x}
{(\tilde f)'y}&\leq C\sum_{i=1}^{k+1}\frac{|\tilde fx-\tilde fy|}{\inf_{\varDelta_{
\omega(i)\cdots\omega(k)}}|(f^{|\omega(i)\cdots\omega(k)|})'|}\leq C\left(\sum_{i=1}^{k}s^{-k+i}+1\right)|\tilde fx-\tilde fy|\\
&\leq C\frac{2s-1}{s-1}|\tilde fx-\tilde fy|.
\end{align*}
The case $k=1$ is covered by \eqref{control1}.
This completes the proof of Proposition \ref{verify}(a).

By Lemma \ref{uniform}, $f$ has uniform decay of cylinders.
Let $\tilde X\subset F^{\mathbb N}$ $(\tilde X\subset X)$ denote the topological Markov shift
determined by $\tilde f$ with alphabet $F$,
and let $\tilde\pi\colon \tilde X\to \bigcap_{n=0}^\infty \tilde f^{-n}K\subset K$ denote the coding map defined
as in Section \ref{CM}, namely
$\tilde\pi( \omega)\in\bigcap_{n=0}^\infty \overline{\tilde f^{-n}\varDelta_{ \omega_n}}$
for $\omega=(\tilde \omega_n)_{n=0}^\infty\in\tilde X$.
Define the  induced potential 
$\tilde\psi\colon\tilde X\to\mathbb R$ by
$\tilde\psi=-\log|D\tilde f\circ\tilde\pi|.$
From Proposition \ref{verify}(a),
$\tilde\psi$ is H\"older continuous. For each $\beta\in\mathbb R$,
denote by $P(\beta\tilde\psi)$ the pressure of 
the potential $\beta\tilde\psi$
(see \eqref{press} or \cite[Section 2.1]{MauUrb03} for the definition). The critical exponent is given by
$$\delta=\inf\{\beta\geq0\colon P(\beta\tilde\psi)\leq0\}.$$

We show $\dim_HJ\le \delta$. Since $(\tilde f)^2$ is uniformly expanding  by Lemma \ref{g}, it follows that $\beta \mapsto P(\beta\tilde\psi)$ is strictly decreasing. Hence, $P((\delta + \epsilon ) \tilde\psi)<0$ for any $\epsilon>0$. By a standard covering argument we verify that the 
Hausdorff $(\delta+\epsilon)$-measure of $\tilde\pi(\tilde X)$ is finite. Namely, for sufficiently large $n \in \mathbb N$, we cover $\tilde\pi(\tilde X)$ by cylinders $ \omega\in F^n$. Since $\tilde \psi$ is H\"older continuous by 
Proposition \ref{verify}(a) and $(\tilde f)^2$ is uniformly expanding, $\tilde \psi$ has the bounded distortion property (cf.  \cite[Lemma 2.3.1]{MauUrb03}). In particular,  $$\sup_{ \omega \in F^n}\frac{|\tilde\pi([\omega])|^{\delta+\epsilon}}{ \exp\left((\delta+\epsilon )\sup \sum_{i=0}^{n-1}\tilde\psi\circ\tilde\sigma^i|_{[\omega]}\right)}<\infty,$$  
where $\tilde\sigma\colon\tilde X\to\tilde X$ is the left shift.
Combining this with $P((\delta + \epsilon ) \tilde\psi)<0$, it follows that  the Hausdorff 
$(\delta+\epsilon)$-measure of $\tilde\pi(\tilde X)$ is finite.
Since $\epsilon>0$ is arbitrary we obtain
$$\dim_HJ=
\dim_H\left(\bigcap_{n=0}^\infty \tilde f^{-n}K\right)\leq\dim_H\tilde\pi(\tilde X)\leq\delta.$$

Next, we verify that $\delta_0\geq\delta.$ To prove this we first make use of the approximation property
of the pressure function $\beta\mapsto P(\beta\tilde\psi)$ by compact invariant subsystems.
Let $\tilde X_n$ denote the finite subshift of $\tilde X$ given by all sequences containing exclusively
the first $n$ symbols of the alphabet of $\tilde X$.
Denote by $P_n(\beta\tilde\psi)$ the corresponding pressure for the potential $\beta\tilde\psi|_{\tilde X_n}$. 
By \cite[Theorem 2.1.6]{MauUrb03}, $P(\beta \tilde\psi)=\lim_{n\to\infty} P_n(\beta\tilde\psi)$ holds. Denote by $\delta_n$ the unique solution
of the equation $P_n(\beta \tilde\psi)=0$. It is then straightforward to verify that 
$\delta=\lim_{n\to\infty}\delta_n$ (see \cite[Theorem 1.7]{JaeKes10}).
Let $\tilde\mu_n$ denote the equilibrium state for the potential $\delta_n\tilde\psi|_{\tilde X_n}$.
Since $\tau$ is integrable against $\tilde\mu_n$,
 the measure $$\mu_n=\left(\sum_{\omega\in F}\tau(\varDelta_\omega)\tilde\mu_n(\varDelta_\omega)\right)^{-1}\sum_{\omega\in F} \sum_{j=0}^{\tau(\varDelta_\omega)-1}f^j_*\tilde\pi_*\tilde\mu_{n}|_{\varDelta_\omega}$$ is in $\mathcal M(f)$, ergodic, expanding and satisfies 
$\dim(\mu_n)=\delta_n$
by Abramov-Kac's formula.
This proves $\delta_0\geq\delta$,
and so Proposition \ref{verify}(b) holds.

The bounded distortion property of $\tilde\psi$ implies
 that $\tilde f$ has mild distortion.
The above approximation by finite subsystems implies \eqref{dim} for $\tilde f$.
This completes the proof of Proposition \ref{verify}(c).
\qed

\begin{remark}
The assumption on the smoothness of $f$ in Proposition \ref{verify} can be relaxed,
so as to accommodate the Manneville-Pomeau type neutral fixed points.
It is enough to assume that a branch containing a neutral fixed point $p$ in its domain is $C^{1+\theta}$, 
$\theta\in(0,1)$, concave near $p$, and that
there exist $A\neq0$, $\beta\in(0,\frac{\theta}{1-\theta})$ such that  
$ f'x= 1+A (x-p)^{\beta}+o((x-p)^{\beta})\quad(x\to p).$
A control of distortion as in Lemma \ref{scope'} holds with
exponent $\theta$,
which one can prove using 
\cite[Lemma 2, Corollary]{Tha83}. 
For details, see \cite{Nak00,Urb96} for example.
\end{remark}

  \section{Mixed Birkhoff spectra associated with the BCF expansion}\label{BCF1}

In this last section,
let $f\colon[0,1)\to[0,1)$
 be the R\'enyi map
with a Markov partition $\{\varDelta_i\}_{i=1}^\infty$,
$\varDelta_i=[1-\frac{1}{i},1-\frac{1}{i+1})$.
We prove the two theorems in the introduction
on mixed Birkhoff spectra for the R\'enyi map.

\subsection{Proof of Theorem \ref{degenerate}}
We need a couple of lemmas.
\begin{lemma}\label{lalpha}
  There exists an ergodic expanding measure with dimension arbitrarily close to $1$.
  In particular, $f$ is saturated.
  Moreover,
 for every $\alpha\in(0,\infty)$,
   \begin{equation*}
   \dim_H L(\alpha)=
  \lim_{\epsilon\to0}\sup\left\{\dim(\mu)\colon \mu\in{\mathcal M}(f),\ \left|\chi(\mu)-\alpha\right|<\epsilon\right\}.\end{equation*}
\end{lemma}
\begin{proof}
Clearly, $f$ satisfies (M3).
  A direct calculation shows that $f$ satisfies R\'enyi's condition, and so $\log|f'|\in\mathcal F$ by Lemma \ref{mild}.
        It is easy to see that the first return map to the interval $(1/2,1)$ is a fully branched Markov map which satisfies all the assumptions in Proposition \ref{verify}. Hence, the first assertion of Lemma \ref{lalpha}.
The second one is a consequence of Main Theorem(b).
\end{proof}

\begin{lemma}\label{seq}
There exists a sequence 
$\{\mu_p\}_{p=1}^\infty$ of expanding measures such that 
$$\lim_{p\to\infty}\dim(\mu_p)=1\ \text{ and }\ \lim_{p\to\infty}\chi(\mu_p)=0,$$
$$\int b_1d\mu_p<\infty\quad\text{for every $p\geq1$}\ \text{ and }\
\lim_{p\to\infty}\int b_1d\mu_p=\infty.$$ 

\end{lemma}
\begin{proof}
 By Lemma \ref{lalpha}, there exists an ergodic expanding measure with dimension
arbitrarily close to $1$. Since such measures are approximated by a finite collection of cylinders 
in the sense of Lemma \ref{katok},
it is possible to take a sequence $\{\xi_p\}_{p=1}^\infty$ of expanding measures
  such that $\chi(\xi_p)\geq p^{-1/4}$
and $\int b_1d\xi_p<\infty$ hold for every $p\geq1$
and $\lim_{p\to\infty}\dim(\xi_p)=1$.
Let $\delta_p$ denote the unit point mass at the fixed point of $f$ in $\varDelta_p$,
and put
$$\mu_p=\left(1-\frac{1}{\sqrt{p}}\right)\xi_p+\frac{1}{\sqrt{p}}\delta_{p}.$$
Then we have $h(\delta_p)=0$,
$\chi(\delta_p)\leq 2\log(p+1)$ and 
\[\int b_1d\mu_p\geq\frac{1}{\sqrt{p}} \int b_1d\delta_p=\frac{p+1}{\sqrt{p}}\rightarrow \infty,\text{ as }p\rightarrow \infty.\]
Moreover,
for any $c>1$ there exists $p_0\geq1$ such that 
$\chi(\mu_p)\leq c\chi(\xi_p)$, for every $p\geq p_0$. Hence
$$\dim(\mu_p)\geq\frac{\left(1-\frac{1}{\sqrt{p}}\right)h(\xi_p)}{c\chi(\xi_p)},$$
which implies
$\liminf_{p\to\infty}\dim(\mu_p)\geq 1/c$. Letting $c\to 1$ yields
$\lim_{p\to\infty}\dim(\mu_p)=1$.

If $\limsup_{p\to\infty}\chi(\mu_p)>0$, then from Main Theorem(b) it follows that
$\dim_HL(\alpha)=1$ for $\alpha=\limsup_{p\to\infty}\chi(\mu_p)$.
This contradicts the fact that the Lyapunov spectrum $\alpha\in[0,\infty)\mapsto\dim_HL(\alpha)$
is strictly monotone decreasing \cite[Theorem 4.2]{Iom10}.
\end{proof}


Let $\{\mu_p\}_{p=1}^\infty$ be a sequence of expanding measures as in Lemma \ref{seq}. 
Let $\alpha\in[2,\infty)$.
For each integer $j\geq1$ fix $p(j)\geq1$ such that
$$\int b_1d\mu_{p(j)}>\alpha+j.$$
Define $t_j\in(0,1]$ implicitly by
$$t_j\int b_1d\mu_{p(j)}+(1-t_j)\int b_1d\delta_1=\alpha+\frac{1}{j}\in\left(2,\alpha+j\right],$$
and put $$\nu_{j}=t_j\mu_{p(j)}+(1-t_j)\delta_1.$$ 
Then $\nu_j$ is an expanding measure and satisfies
$\int b_1d\nu_j=\alpha+1/j$.
Note that we have $t_j\to0$ as  $j\to\infty$.

Let $\boldsymbol\phi=(\phi_i)_{i=1}^\infty\in\mathcal F^{\mathbb N}$ be such that $\phi_i$ is bounded (not necessarily continuous) for every $i\geq1$. Put $\boldsymbol\alpha=(\phi_i(0))_{i=1}^\infty\in\mathbb R^{\mathbb N}$.
By a diagonal argument, one can choose a subsequence $\{\nu_{j_\ell}\}_{\ell=1}^\infty$
of $\{\nu_j\}_{j=1}^\infty$ such that
 $\lim_{\ell\to\infty}\left\|\int \boldsymbol\phi_k d\nu_{j_\ell}-\boldsymbol\phi_k(0)\right\|=0$ for each $k\geq1$.
Since $h(\delta_1)=0=\chi(\delta_1)$, 
$\dim(\nu_j)=\dim(\mu_{p(j)})$ holds.
For each $k\geq1$,
Proposition \ref{lowd} yields \begin{align*}\lim_{\epsilon\to0}\sup&\left\{\dim(\mu)\colon
\mu\in\mathcal M(f),\
\left|\int\boldsymbol\phi_kd\mu-\boldsymbol\alpha_k\right|<\epsilon,\left|\int b_1d\mu-\alpha\right|<\epsilon\right\}\\
&\geq\lim_{\ell\to\infty}\dim(\nu_{j_\ell})=\lim_{\ell\to\infty}\dim(\mu_{p(j_\ell)})=1.\end{align*}
Since $k\geq1$ is arbitrary, the Hausdorff dimension of the level set is bounded from below by $1$, and so it is $1$.
In the case $\alpha=\infty$,
for each $p\geq2$ define an expanding measure $\nu_{p}=(1-\frac{1}{p})\mu_{p}+\frac{1}{p}\delta_1.$
From Lemma \ref{seq}, 
$\lim_{p\to\infty}\int b_1d\nu_p=\infty$.
A slight modification of the proof of Proposition \ref{lowd} and the same reasoning in the case 
where $\alpha$ is finite
shows that the Hausdorff dimension of the level set is bounded from below by
$\lim_{p\to\infty}\dim(\nu_{p})=\lim_{p\to\infty}\dim(\mu_{p})=1.$
This completes the proof of Theorem \ref{degenerate}.
\qed

  

\subsection{Proof of Theorem \ref{dimthm}}\label{pthmb}
The dimension formula in the first assertion is a consequence of Main Theorem(b).
We prove the second assertion.
 From Lemma \ref{lalpha} and $\lim_{\alpha\to\infty}\dim_HL(\alpha)=1/2$ as in
    \cite[p.219]{Iom10} we have
 \begin{equation}\label{atinfty}
   \limsup_{\alpha\to\infty}\sup\left\{\dim(\mu)\colon\mu\in\mathcal M(f),\ \chi(\mu)=\alpha\right\}\leq\frac{1}{2}.
   \end{equation}
Let $\boldsymbol\alpha\in\mathbb R^{\mathbb N}$ be a frequency vector
such that $\sum_{i=1}^\infty\alpha_i<1$. 
Fix $\epsilon_0\in(0,1-\sum_{i=1}^\infty\alpha_i)$ and
 put $c=1-\sum_{i=1}^\infty\alpha_i-\epsilon_0>0.$
  For each integer $k\geq 1$ take $\epsilon(k)\in(0,\epsilon_0/k)$.
  Then
  \begin{equation}\label{sev}
  \begin{split}\lim_{\epsilon\to0}&\sup\left\{\dim(\mu)\colon \mu\in\mathcal M(f),\
   \max_{1\leq i\leq k}\left|\mu(\varDelta_i) -\alpha_i\right|<\epsilon\right\}\\
   &\leq
   \sup\left\{\dim(\mu)\colon \mu\in\mathcal M(f),\
   \max_{1\leq i\leq k}\left|\mu(\varDelta_i) -\alpha_i\right|<\epsilon(k)\right\}\\
   &\leq \sup\left\{\dim(\mu)\colon \mu\in\mathcal M(f),\
   \sum_{i=k+1}^\infty\mu(\varDelta_i)>c\right\},
   \end{split}
   \end{equation}
   where the last inequality follows from
   $$1-\sum_{i=1}^k\alpha_i-k\epsilon(k)>1-\sum_{i=1}^k
   \alpha_i-\epsilon_0\geq c>0.$$
Since $\lim_{i\to\infty}\inf_{\varDelta_i} |f'|=\infty$ we have
   \begin{equation}\label{equat}\lim_{k\to\infty}\inf\left\{\chi(\mu)\colon\mu\in\mathcal M(f),\ \sum_{i=k+1}^\infty\mu(\varDelta_i)>c\right\}=\infty.\end{equation}
 By \eqref{atinfty} and \eqref{equat}, letting $k\to\infty$ in \eqref{sev} proves
\begin{equation*} \label{eq:binfty}
   \lim_{k\to \infty}\lim_{\epsilon\to0}\sup\left\{\dim(\mu)\colon \mu\in\mathcal M(f),\
   \max_{1\leq i\leq k}\left|\mu(\varDelta_i) -\alpha_i\right|<\epsilon\right\}\le \frac{1}{2}.
   \end{equation*}
   This inequality and
  the dimension formula together imply $\dim_HBE(\boldsymbol\alpha)=1/2$.\qed

\subsection*{Acknowledgments}
The first-named author was partially supported by the JSPS KAKENHI 17K14203. 
The second-named author was partially supported by the JSPS KAKENHI 
19K21835, 20H01811.

\end{document}